\documentclass[]{siamart190516}
\usepackage[utf8]{inputenc}
\usepackage{amsfonts,amsopn,braket}
\usepackage{amssymb}
\usepackage{color}
\usepackage{booktabs}
\usepackage{mcode}
\usepackage{graphicx}
\usepackage{comment}
\usepackage{tikz}
\usepackage{pgfplots}
\usepackage{algorithm}
\usepackage{algpseudocode}

\title{On robustly convergent and efficient iterative methods for anisotropic radiative transfer
\thanks{Submitted to the editors \today.}}
\author{J\"urgen D\"olz%
\thanks{Institute for Numerical Simulation, University of Bonn, Friedrich-Hirzebruch-Allee 7, 53115 Bonn, Germany.
\email{doelz@ins.uni-bonn.de}}
\and Olena Palii\thanks{Department of Applied Mathematics, University of Twente,
P.O. Box 217, 7500 AE Enschede, The Netherlands.
\email{o.palii@utwente.nl}}
\and Matthias Schlottbom%
\thanks{Department of Applied Mathematics, University of Twente,
P.O. Box 217, 7500 AE Enschede, The Netherlands.
\email{m.schlottbom@utwente.nl}}
}

\newcommand{\TheTitle}{Iterative methods for anisotropic RTE}
\newcommand{\TheAuthors}{J. D\"olz, O. Palii, M. Schlottbom}
\headers{\TheTitle}{\TheAuthors}

\def\s{s}
\def\r{r}

\def\n{n}
\def\sn{s \cdot \n}
\def\sgrad{s \cdot \nabla_\r}
\def\u{u}

\def\K{K}

\def\H{H}
\def\q{q}

\def\A{\mathcal{A}}

\def\E{\mathcal{E}}
\def\H{\mathcal{H}}
\def\I{\mathcal{I}}
\def\K{\mathcal{K}}

\def\M{\mathcal{M}}

\def\cO{\mathcal{O}}
\def\R{\mathcal{R}}

\def\S{\mathcal{S}}
\def\dD{\partial D}

\def\P{\mathcal{P}}
\def\T{\mathcal{T}}

\def\SS{\mathbb{S}}
\def\XX{\mathbb{X}}

\def\RR{\mathbb{R}}

\def\NN{\mathbb{N}}
\def\PP{\mathbb{P}}
\def\WW{\mathbb{W}}
\def\VV{\mathbb{V}}

\def\ThR{\mathcal{T}_h^R}
\def\ThS{\mathcal{T}_h^S}

\def\Id{\mathcal{I}}

\def\ue{u^{\!+}}
\def\uo{u^{\!-}}

\def\bA{\mathbf{A}}
\def\bB{\mathbf{B}}
\def\bC{\mathbf{C}}
\def\bE{\mathbf{E}}

\def\bI{\mathbf{I}}

\def\bK{\mathbf{K}}

\def\bM{\mathbf{M}}

\def\bP{\mathbf{P}}

\def\bR{\mathbf{R}}

\def\bU{\mathbf{U}}

\def\bW{\mathbf{W}}
\def\bX{\mathbf{X}}

\def\bb{\mathbf{b}}

\def\bf{\mathbf{f}}

\def\bq{\mathbf{q}}
\def\br{\mathbf{r}}
\def\bs{\mathbf{s}}

\def\bu{\mathbf{u}}
\def\bv{\mathbf{v}}
\def\bw{\mathbf{w}}
\def\bx{\mathbf{x}}
\def\by{\mathbf{y}}
\def\bz{\mathbf{z}}

\def\bAi{\boldsymbol{\mathsf{A}}}
\def\bIdSpheree{\boldsymbol{\mathsf{I}}\!^{+}}
\def\bIdSphereo{\boldsymbol{\mathsf{I}}\!^{-}}

\def\bMSpheree{\boldsymbol{\mathsf{M}}\!^{+}}
\def\bMSphereo{\boldsymbol{\mathsf{M}}\!^{-}}
\def\bSe{\boldsymbol{\mathsf{S}}\!^{+}}
\def\bSo{\boldsymbol{\mathsf{S}}\!^{-}}
\def\bLambda{\boldsymbol{\mathsf{\Lambda}}}
\def\bWSphereo{\boldsymbol{\mathsf{W}}\!^{-}}
\def\bWSpheree{\boldsymbol{\mathsf{W}}\!^{+}}

\def\bD{\boldsymbol{\mathfrak{D}}}
\def\bMse{\boldsymbol{\mathfrak{M}}\!^{+}}
\def\bMso{\boldsymbol{\mathfrak{M}}\!^{-}}
\def\bIdso{\boldsymbol{\mathfrak{I}}\!^{-}}
\def\bIdse{\boldsymbol{\mathfrak{I}}\!^{+}}
\def\bRs{\boldsymbol{\mathfrak{R}}}

\def\bIde{\bI^{\!+}}

\def\bMe{\bM^{\!+}}
\def\bMo{\bM^{\!-}}
\def\bKe{\bK^{\!+}}
\def\bKo{\bK^{\!-}}
\def\bWe{\bW^{\!+}}

\def\ble{\bq^{\!+}}
\def\blo{\bq^{\!-}}
\def\bue{\bu^{\!+}}
\def\buo{\bu^{\!-}}

\def\bUe{\bU^{\!+}}
\def\bUo{\bU^{\!-}}

\def\st{\sigma_t}
\def\sa{\sigma_a}
\def\ss{\sigma_s}

\def\bfx{\mathbf{x}}
\def\bfy{\mathbf{y}}

\newtheorem{remark}[theorem]{Remark}

\newcommand{\tnorm}[1]{{\left\vert\kern-0.25ex\left\vert\kern-0.25ex\left\vert#1\right\vert\kern-0.25ex\right\vert\kern-0.25ex\right\vert}}

\begin{document}
\maketitle

\begin{abstract}
This paper considers the iterative solution of linear systems arising from discretization of the anisotropic radiative transfer equation with discontinuous elements on the sphere. 
In order to achieve robust convergence behavior in the discretization parameters and the physical parameters we develop preconditioned Richardson iterations in Hilbert spaces. We prove convergence of the resulting scheme. The preconditioner is constructed in two steps. The first step borrows ideas from matrix splittings and ensures mesh independence. The second step uses a subspace correction technique to reduce the influence of the optical parameters. The correction spaces are build from low-order spherical harmonics approximations generalizing well-known diffusion approximations.
We discuss in detail the efficient implementation and application of the discrete operators. In particular, for the considered discontinuous spherical elements, the scattering operator becomes dense and we show that $\mathcal{H}$- or $\mathcal{H}^2$-matrix compression can be applied in a black-box fashion to obtain almost linear or linear complexity when applying the corresponding approximations. 
The effectiveness of the proposed method is shown in numerical examples.
\end{abstract}

\begin{keywords}
anisotropic radiative transfer, iterative solution, preconditioning, compression
\end{keywords}

\begin{AMS}
65F08 
65F10 
65N22 
65N30 
65N45 
\end{AMS}
 
\section{Introduction}
Radiative transfer models describe the streaming, absorption, and scattering of radiation waves propagating through a turbid medium occupying a bounded convex domain $R\subset\RR^d$, and they arise in a variety of applications, e.g., neutron transport~\cite{CaseZweifel67,LewisMiller84}, heat transfer~\cite{Modest}, climate sciences~\cite{evans1998spherical}, geosciences~\cite{Meng_2017} or medical imaging and treatment~\cite{Ahmedov_2016,Arridge_2009,Tarvainen_2017}. 
 The underlying physical model can be described by the anisotropic radiative transfer equation,
\begin{align}
\sgrad \u(\s,\r) + \sigma_t(\r) \u(\s,\r) &= \sigma_s(\r)\int_S k(\s\cdot \s')\u(\s',\r)d\s' + \q(\s,\r).\label{eq:rte1} 
\end{align}
The specific intensity $\u=\u(\s,\r)$ depends on the position $\r\in R$ and the direction of propagation described by a unit vector $\s\in S$, i.e., we assume a constant speed of propagation.
The medium is characterized by the total attenuation coefficient $\sigma_t=\sigma_a+\sigma_s$, where $\sa$ and $\ss$ denote the absorption and scattering rates, respectively. The scattering phase function $k$ relates pre- and post-collisional directions, and we consider exemplary the Henyey-Greenstein phase function
\begin{align}\label{eq:HG-kernel1}
    k(s\cdot s')=\frac{1}{4\pi}\frac{1-g^2}{{[1-2g(s\cdot s')+g^2]}^{3/2}},
\end{align}
with anisotropy factor $g$. For $g=0$, we speak about isotropic scattering, and for $g$ close to one, we say that the scattering is (highly) forward peaked. For simplicity, we assume $0\leq g<1$ in the following. The case $-1<g\leq 0$ is similar.
Internal sources of radiation are modeled by the function $\q$.
Introducing the outer unit normal vector field $\n(\r)$ on $\partial R$, the boundary condition is modeled by
\begin{align}
    \u(\s,\r) = g(\s,\r)  \quad \text{for } (\s,\r)\in S\times\partial R \text{ such that } s\cdot \n(\r)<0. \label{eq:rte2}
\end{align}
In this paper we consider the iterative solution of the linear systems arising from the discretization of the anisotropic radiative transfer equations \cref{eq:rte1}--\cref{eq:rte2} by preconditioned Richardson iterations. 
We are particularly interested in robustly convergent methods for multiple physical regimes that, at the same time, can embody ballistic regimes $\ss\ll 1$ and diffusive regimes, i.e., $\ss\gg 1$ and $\sa>0$, and highly forward peaked scattering, as it occurs for example in medical imaging applications~\cite{Gonzalez_2008}.
Due to the size of the arising systems of linear equations, their numerical solution is challenging, and a variety of methods were developed as briefly summarized next. 
\subsection{Related work}
Since for realistic problems analytical solutions are not available, numerical approximations are required.
Common discretization methods can be classified into two main approaches based on their semidiscretization in $\s$. The spherical harmonics method \cite{AydinOliveiraGoddard04,egger2019,LewisMiller84} approximates the solution $\u$ by a truncated series of spherical harmonics, which allows for spectral convergence for smooth solutions. For non-smooth solutions, which is the generic situation, local approximations in $\s$ can be advantageous, which is achieved, e.g., by discrete ordinates methods \cite{Han2010,LewisMiller84,Sun2020,Wang:2018,Tano2021}, continuous Galerkin methods \cite{becker2010finite}, the discontinuous Galerkin (DG) method \cite{GuermondKanschatRagusa2014,kophazi2015space,palii2020convergent},
iteratively refined piecewise polynomial approximations \cite{Dahmen_2020}, or hybrid methods \cite{Crockatt2020,Heningburg2020}.

A common step in the solution of the linear systems resulting from local approximations in $\s$ is to split the discrete system into a transport part and a  scattering part. While the inversion of transport is usually straight-forward, scattering introduces a dense coupling in $\s$. The corresponding Richardson iteration resulting from this splitting is called the source iteration \cite{AdamsLarsen02,MarchukLebedev86}, and it converges linearly with a rate $c=\|\ss/\st\|_\infty$. For scattering dominated problems, such as the biomedical applications mentioned above, we have $c\approx 1$ and the convergence of the source iteration becomes too slow in such applications. Acceleration of the source iteration can be achieved by preconditioning, which usually employs the diffusion approximation to \eqref{eq:rte1}--\eqref{eq:rte2} \cite{AdamsLarsen02}, and the resulting scheme is then called diffusion synthetic accelerated (DSA) source iteration \cite{AdamsLarsen02}. Although this approach is well motivated by asymptotic analysis, it faces several issues, such as, a proper generalization to multi-dimensional problems with anisotropy, strong variations in the optical parameters, or the use of unstructured and curved meshes, see \cite{AdamsLarsen02}.

Effective DSA schemes rely on consistent discretization of the corresponding diffusion approximation, see \cite{palii2020convergent,Warsa2002} for isotropic scattering, and in \cite{Ragusa2010} for two-dimensional problems with anisotropic scattering. The latter employs a modified interior penalty DG discretization for the corresponding diffusion approximation, which has also been used in \cite{Wang2010} where it is, however, found that their DSA scheme becomes less effective for highly heterogeneous optical parameters. A discrete analysis of DSA schemes for high-order DG discretizations on possibly curved meshes, which may complicate the inversion of the transport part, can be found in \cite{Haut2020}.
In the variational framework of \cite{palii2020convergent} consistency is automatically achieved by subspace correction instead of finding a consistent discretization of the diffusion approximation. This variational treatment allowed to prove convergence of the corresponding iteration and numerical results showed robust contraction rates, even in multi-dimensional calculations with heterogeneous optical parameters. 

The subspace correction approach can be related to multigrid schemes \cite{Xu_2002}, and we refer to \cite{Kanschat_2014,Lee2012,Shao2020} and the references there in the context of radiative transfer. Comparing to non-symmetric Krylov space methods, such as GMRES or BiCGStab, see \cite{AdamsLarsen02,badri2019preconditioned,warsa2002krylov} and the references there, our approach is very memory effective and monotone convergence behavior is guaranteed. Moreover, in view of its good convergence rates, the considered preconditioned Richardson iteration is competitive to these multilevel and Krylov space methods.
It is the purpose of this paper to generalize the approach of \cite{palii2020convergent} to the anisotropic scattering case, which requires non-trivial extensions as outlined in the next section.

\subsection{Approach and contribution}\label{sec:contribution}
In this paper we focus on the construction of \emph{robustly and provably convergent efficient iterative schemes} for the radiative transfer equation with anisotropic scattering. To describe our approach, let us introduce the linear system that we need to solve, which stems from a mixed finite element discretization of \eqref{eq:rte1}--\eqref{eq:rte2} using discontinuous polynomials on the sphere \cite{egger2012mixed,palii2020convergent}, i.e., 
\begin{align}\label{eq:mixed_discrete_intro}
    \begin{bmatrix}\bR + \bMe & -\bA\!^{\intercal}\\ \bA & \bMo\end{bmatrix}\begin{bmatrix}\bue\\ \buo\end{bmatrix} = \begin{bmatrix} \bKe&\\ &\bKo \end{bmatrix}\begin{bmatrix}\bue\\ \buo\end{bmatrix} + \begin{bmatrix}\ble\\ \blo\end{bmatrix}.
\end{align}
Here, the superscripts in the equation refer to even (`$+$') and odd (`$-$') parts from the underlying discretization. The matrices $\bKe$ and $\bKo$ discretize scattering, while $\bR$ incorporates boundary conditions, $\bMe$ and $\bMo$ are mass matrices related to $\st$, and $\bA$ discretizes $\s\cdot\nabla_\r$, and their assembly can be done with standard FEM codes.
The even part solves the even-parity equations
\begin{align}\label{eq:ep_discrete_intro}
    \bE \bue = \bKe \bue + \bq,
\end{align}
i.e., the Schur complement of \eqref{eq:mixed_discrete_intro},
with symmetric positive definite matrix $\bE=\bA\!^{\intercal} (\bMo-\bKo)^{-1}\bA+\bMe+\bR$ and source term $\bq=\ble+\bA\!^{\intercal} (\bMo-\bKo)^{-1}\blo$. 
Once the even part $\bue$ is known, the odd part $\buo$ can be obtained from \cref{eq:mixed_discrete_intro}. 
The preconditioned Richardson iteration considered in this article then reads
\begin{align}\label{eq:richardson_intro}
    \bue_{n+1}= \big(\bI - \bP_2 \bP_1 (\bE-\bKe)\big)\bue_n + \bP_2\bP_1 \bq,
\end{align}
with preconditioners $\bP_1$ and $\bP_2$. Comparing to standard DSA source iterations, $\bP_1$ corresponds to a transport sweep, and a typical choice that renders the convergence behavior of \cref{eq:richardson_intro} independent of the discretization parameters is $\bP_1=\bE^{-1}$. More precisely, we show that this choice of $\bP_1$ yields a contraction rate of $c=\|\ss/\st\|_\infty$.
The second preconditioner $\bP_2$ aims to improve the convergence behavior in diffusive regimes, $c\approx 1$.
In the spirit of \cite{palii2020convergent}, we construct $\bP_2$ via Galerkin projection onto suitable subspaces, which guarantees monotone convergence of \cref{eq:richardson_intro}.
The construction of suitable subspaces that give good error reduction is motivated by the observation that error modes that are damped hardly by $\bI-\bP_1(\bE-\bKe)$ can be approximated well by spherical harmonics of low degree. While for the isotropic case $g=0$, spherical harmonics of degree zero, i.e., constants in angle, are sufficient for obtaining good convergence rates, we show that higher order spherical harmonics should be used for anisotropic scattering. To preserve consistency, we replace higher order spherical harmonics, which are the eigenfunctions of the integral operator in \cref{eq:rte1}, by discrete eigenfunctions of $\bKe$.

The efficiency of the proposed iterative scheme hinges on the ability to efficiently implement and apply the arising operators.
While for $g=0$, $\bKo=0$, and $\bKe$ can be realized via fast Fourier transformation, and $\bE$ is block-diagonal with sparse blocks allowing for an efficient application of $\bE$, the situation is more involved for $g>0$.
We show that $\bKe$ and $\bKo$ can be applied efficiently by exploiting their Kronecker structure between a sparse matrix and a dense matrix which turns out to be efficiently applicable by using $\mathcal{H}$- or $\mathcal{H}^2$-matrix approximations independently of $g$. 
As we show the practical implementation of $\mathcal{H}$- or $\mathcal{H}^2$-matrices can be done by standard libraries, such as \textsc{H2LIB} \cite{Boe} or \textsc{BEMBEL} \cite{DHK+2020}. This in combination with standard FEM assembly routines for the other matrices ensures robustness and maintainability of the code.

Since $\bA$, $\bMe$, and $\bR$ are sparse and block diagonal, the main bottleneck in the application of $\bE$ is the application of $(\bMo-\bKo)^{-1}$. Based on the tensor structure of $\bKo$ and its spectral properties, we derive a preconditioner such that $(\bMo-\bKo)^{-1}$ can be applied robustly in $g$ in only a few iterations. 
Thus, we can apply $\bE$ in almost linear complexity.
Efficiency of \cref{eq:richardson_intro} is further increased by realizing $\bP_1=\bE^{-1}\approx \bP_1^l$ inexactly by employing a small, fixed number of $l$ steps of an inner iterative scheme. We show that the condition number of $\bP_1^l\bE$ is $O((cg)^{-l})$ which is robust in the limit $c\to 1$.
In contrast, we note that the condition number of $\bP_1^l(\bE-\bKe)$ is $O((1-c)^{-1})$, i.e., a straight-forward iterative solution of the even-parity equations using a black-box solver, such as preconditioned conjugate gradients, is in general not robust for $c\to 1$.

Summarizing, each step of our iteration \cref{eq:richardson_intro} can be performed very efficiently. The iteration is provably convergent and numerical results show that the contraction rates are robust for $c\to 1$. The result is a highly efficient numerical scheme for the solution of the even parity equations \cref{eq:ep_discrete_intro} and, thus, also for the overall system \cref{eq:mixed_discrete_intro}.

\subsection{Outline}
The structure of the paper is as follows: In \cref{sec:preliminaries} we recall the variational formulation that builds the basis of our numerical scheme and establish some spectral equivalences for the scattering operator which are key to the construction of our preconditioners. In \cref{sec:iteration_ep} we present iterative schemes for the even-parity equations of radiative transfer in Hilbert space, which, after discretization in \cref{sec:galerkin}, result in the schemes described in \cref{sec:contribution}. Details of the implementation and its complexity are described in \cref{sec:richardson_discrete}.
Numerical studies of the performance of the proposed methods and report on the results are presented in \cref{sec:numerics}. The paper closes with a discussion in \cref{sec:discussion}.

\section{Preliminaries}\label{sec:preliminaries}
In the following we recall the relevant functional analytic framework, state the corresponding variational formulation of the radiative transfer problem \eqref{eq:rte1}--\eqref{eq:rte2} and provide some analytical results about the spectrum of the scattering operator which we will later use for the construction of our preconditioners.

\subsection{Function spaces}
By $L^2(M)$ we denote the usual Hilbert space of square integrable functions on a manifold $M$, and denote $(u,w)_M=\int_{M} uw\, dM$ the corresponding inner product and $\|u\|_{L^2(M)}$ the induced norm. For $M=D=S\times R$, we write $\VV=L^2(D)$ and $(u,w)=(u,w)_D$. Functions $w\in \VV$ with weak derivative $\sgrad w\in \VV$ have a well-defined trace \cite{ManResSta00}. We restrict the natural trace space \cite{ManResSta00}, and consider the weighted Hilbert space $L^2(\dD_\pm;|\sn|)$ of measurable functions $w$ on $\dD_\pm=\{(\s,\r)\in S\times\partial R: \pm s\cdot n(r)>0\}$ with $|\sn|^{1/2} w\in L^2(\dD_\pm)$.
For the weak formulation of \cref{eq:rte1}--\cref{eq:rte2} we use the Hilbert space
\begin{align*}
    \WW=\{w\in L^2(D): \sgrad w\in L^2(D),\,\, w_{\mid\dD_-}\in L^2(\dD_-;|\sn|)\},
\end{align*}
with corresponding norm $\|w\|_\WW^2=\|\sgrad w\|_{L^2(D)}^2+\|w\|_{L^2(D)}^2+\|w\|_{L^2(\dD_-;|\sn|)}^2$.
\subsection{Assumptions on the optical parameters and data}\label{sec:parameters}
The data terms are assumed to satisfy $q \in L^2(D)$ and $g\in L^2(\dD_-;|\sn|)$. Absorption and scattering rates are non-negative and essentially bounded functions $\sa,\ss\in L^\infty(R)$. We assume that the medium occupied by $R$ is absorbing, i.e., that there exists a constant $\gamma>0$ such that $\sa(\r)\geq \gamma$ for a.e.\ $\r\in R$. Thus, the ratio between the scattering rate and the total attenuation rate $\st=\sa+\ss$ is strictly less than one, $c=\|\ss/\st\|_\infty<1$.
\subsection{Even-odd splitting}
The space $\VV=\VV^+\oplus \VV^-$ allows for an orthogonal decomposition into even and odd functions of the variable $s\in S$. The even part $\ue$ and odd part $\uo$ of a function $u\in\VV$ is defined a.e.\ by $u^{\pm}(\s,\r)=\frac{1}{2}(u(\s,\r)\pm u(-\s,\r))$.
Similarly, we denote $\WW^{\pm}$ the corresponding subspaces of functions $u\in\WW$ with $u\in \VV^\pm$.
\subsection{Operator formulation of the radiative transfer equation}
The weak formulation of \cref{eq:rte1}--\cref{eq:rte2} presented in \cite{egger2012mixed} can be stated concisely using suitable operators and we refer to \cite{egger2012mixed} for proofs of the corresponding mapping properties.
Let $\ue,w^+\in\WW^+$ and $\uo\in\VV^-$.
The transport operator $\A:\WW^+\to\VV^-$ is defined by $\A \ue=\sgrad \ue$. Identifying the dual $\VV'$ of $\VV$ with $\VV$, the dual transport operator $\A':\VV^-\to (\WW^+)'$ is defined by $\langle \A' \uo,w^+\rangle=(\A w^+, \uo)$. Boundary terms are handled by the operator $\R:\WW^+\to (\WW^+)'$ defined by $\langle \R \ue,w^+\rangle=(|\sn|\ue,w^+)_{\dD}$.
Scattering is described by the operator $\S:L^2(S)\to L^2(S)$ defined by
\begin{align*}
    (\S u)(\s) = \int_{S} k(\s\cdot\s') u(\s')d\s',
\end{align*}
where $k$ is the phase function defined in \cref{eq:HG-kernel1}. In slight abuse of notation, we also denote the trivial extension of $\S$ to an operator $L^2(D)\to L^2(D)$ by $\S$. We recall that $\S$ maps even to even and odd to odd functions \cite[Lemma~2.6]{egger2012mixed}, and so does $\K:\VV\to\VV$ defined by $\K u = \ss\S u$. We denote by $\K$ also its restrictions to $\VV^\pm$ and $\WW^+$, respectively.
The spherical harmonics $\{H^l_m: l\in\NN_0, -l\leq m\leq l\}$ form a complete orthogonal system for $L^2(S)$, and we assume the normalization $\|H^l_m\|_{L^2(S)}=1$. Furthermore, $H^{l}_m$ is an eigenfunction of $\S$ with eigenvalue $g^l$, i.e.,
\begin{align}\label{eq:eigS}
\S H^l_m = g^l H^l_m,
\end{align}
and $H^l_m\in\VV^+$ if $l$ is an even number and $H^l_m\in \VV^-$ if $l$ is an odd number.
Attenuation is described by the multiplication operator $\M:\VV\to\VV$ defined by $\M u=\st u$.
Introducing the functionals $\ell^+ \in (\WW^+)'$, $\ell^+(w^+)=(q,w^+) +2 (|\sn| g, w^+)_{\dD_-}$ and $\ell^-\in (\VV^-)'$, $\ell^-(w^-)=(q,w^-)$, the operator formulation of the radiative transfer equation \cref{eq:rte1}--\cref{eq:rte2} is \cite{egger2012mixed}: 
Find $(\ue,\uo)\in\WW^+\times\VV^-$ such that
\begin{align}
    \R \ue - \A' \uo + \M \ue &= \K \ue +\ell^+ \qquad\text{ in } (\WW^+)', \label{eq:op1}\\
    \A \ue + \M \uo &= \K \uo + \ell^- \qquad \text{ in } \VV^-.\label{eq:op2}
\end{align}
\subsection{Well-posedness}
In the situation of \cref{sec:parameters}, there exists a unique solution $(\ue,\uo)\in\WW^+\times\VV^-$ of \cref{eq:op1,eq:op2} satisfying
\begin{align*}
    \|\ue\|_{\WW} + \|\uo\|_{\VV} \leq C (\|q\|_{L^2(D)}+ \|g\|_{L^2(\dD_-;|\sn|)}),
\end{align*}
with a constant $C$ depending only on $\gamma$ and $\|\st\|_\infty$ \cite{egger2012mixed}.
Notice that this well-posedness result remains true even if $\sa$ and $\ss$ are allowed to vanish \cite{EggerSchlottbom14}.
As shown in \cite[Theorem 4.1]{egger2012mixed} it holds that $\uo\in\WW^-$ and $\ue+\uo\in\WW$ satisfies \cref{eq:rte1} a.e.\ in $D$ and \cref{eq:rte2} holds in $L^2(\dD_-;|\sn|)$.
\subsection{Even-parity formulation}\label{sec:ep}
As in \cite{egger2012mixed}, it follows from \cref{eq:eigS} that
\begin{align}\label{eq:equivalence_low}
    \inf_{r\in R} (\sa+(1-g)\ss) \|v^-\|_\VV^2 \leq \|v^-\|_{\M-\K}^2\leq \|\st\|_\infty \|v^-\|^2_\VV\quad\text{for } v^-\in\VV^-,
\end{align}
where we write $\|w\|_{\mathcal{Q}}^2=(\mathcal{Q}w,w)$ for any positive operator $\mathcal{Q}$.
Thus, $\M-\K:\VV^-\to \VV^-$ is boundedly invertible, and, by \cref{eq:op2},
\begin{align}\label{eq:getv}
    \uo= (\M-\K)^{-1} (\ell^- -\A \ue).
\end{align}
Using \cref{eq:getv} in \cref{eq:op1} and introducing 
\[
\E:\WW^+\to(\WW^+)',\quad \E \ue = \R \ue + \A' (\M-\K)^{-1} \A \ue + \M \ue,
\]
and $\ell(w^+)=\ell^+(w^+) +((\M-\K)^{-1} q,\A w^+)$ for $w^+\in\WW^+$,
the even-parity formulation of the radiative transfer equation is: Find $\ue\in\WW^+$ such that
\begin{align}\label{eq:ep}
    (\E - \K) \ue = \ell.
\end{align}
As shown in \cite{egger2012mixed}, the even-parity formulation is a coercive, symmetric problem, which is well-posed by the Lax-Milgram lemma. Solving \cref{eq:ep} for $\ue\in\WW^+$, we can retrieve $\uo\in\VV^-$ by \cref{eq:getv}. In turn, $(\ue,\uo)\in\WW^+\times \VV^-$ solves \cref{eq:op1}--\cref{eq:op2}.

\subsection{Preconditioning of $\M-\K$}\label{sec:pre}
We generalize the inequalities \cref{eq:equivalence_low} to obtain spectrally equivalent approximations to $\M-\K$. Since $\K=\ss\S$, we can construct approximations to $\K$ by approximating $\S$. To do so let us define for $N\in\NN$ and $v\in\VV$
\begin{align}
 \S_N v =\sum_{l=0}^{N} g^{l} \sum_{m=-l}^{l} (v,H^l_m)_{S} H^l_m.
\end{align}
Notice that the summation is only over even integers $0\leq l\leq N$ if $v\in\VV^+$ and only over odd ones if $v\in\VV^-$. The approximation of $\K$ is then defined by $\K_N=\ss\S_N$.
\begin{lemma}\label{lem:spectral_equivalence_inf}
The operator $\M-\K_N$ is spectrally equivalent to $\M-\K$, that is
\begin{align*}
  \big(1-cg^{N+1}\big)  ((\M-\K_N)v,v)\leq 
    ((\M-\K)v,v)
    \leq
    ((\M-\K_N)v,v)
\end{align*}
for all $v\in\VV$,
with $c=\|\ss/\st\|_\infty$. In particular, $\M-\K_N$ is invertible.
\end{lemma}
\begin{proof}
	We use that $\{H^m_l\}$ is a complete orthonormal system of $L^2(S)$. Hence, any $v\in \VV=L^2(S)\otimes L^2(R)$ has the expansion
	\begin{align*}
	    v(\s,\r) = \sum_{l=0}^\infty \sum_{m=-l}^l v^l_m(\r) H^l_m(\s),
	\end{align*}
	with $v^l_m\in L^2(R)$ and $\|v\|_{\VV}^2=\sum_{l=0}^\infty \sum_{m=-l}^l \|v^l_m\|^2_{L^2(R)}<\infty$, and
	\begin{align*}
	    ((\M-\K_N)v,v) = \sum_{l=0}^L\sum_{m=-l}^l \Big\|\sqrt{\st-g^l\ss}v^l_m\Big\|_{L^2(R)}^2 + \sum_{l=N+1}^\infty\sum_{m=-l}^l \Big\|\sqrt{\st} v^l_m\Big\|_{L^2(R)}^2.
	\end{align*}
	Using $c=\|\ss/\st\|_\infty$ it follows that
	\begin{align}\label{eq:error_K}
        0\leq ( (\K-\K_N)v,v)&=\sum_{l=N+1}^\infty g^{l}\sum_{m=-l}^l \Big\|\sqrt{\ss} v^l_m\Big\|_{L^2(R)}^2\leq c g^{N+1} ((\M-\K_N)v,v).
    \end{align}
The inequalities in the statement then follow from
$((\M-\K)v,v) = ((\M-\K_N)v,v) - ((\K-\K_N)v,v)$, while
invertibility follows from \cite[Lemma 2.14]{egger2012mixed}.
\end{proof}
\section{Iteration for the even-parity formulation}\label{sec:iteration_ep}
We generalize the Richardson iteration of \cite{palii2020convergent} for the radiative transfer equation with isotropic scattering to the anisotropic case and equip them with a suitable preconditioner which we will investigate later. We restrict ourselves to a presentation suitable for the error analysis and postpone the linear algebra setting and the discussion of its efficient realization to \cref{sec:richardson_discrete}.
We consider the solution of \cref{eq:ep} along the following two steps: 

\noindent
\textbf{Step (i)} Given $\ue_n\in\WW^+$ and a symmetric and positive definite operator $\P_1:(\WW^+)'\to \WW^+$, we compute
\begin{align}\label{eq:richardson_general}
\ue_{n+\frac{1}{2}}=\ue_n-\P_1((\E-\K)\ue_n-\ell).
\end{align} 
\textbf{Step (ii)} Compute a subspace correction to $\ue_{n+1/2}$ based on the observation that the error $e^+_{n+1/2}=\ue-\ue_{n+1/2}$ satisfies
\begin{align}\label{eq:ep_error_eq}
 (\E-\K) e^+_{n+\frac{1}{2}} = ((\E-\K)\P_1-\Id)((\E-\K)\ue_n-\ell).
\end{align}
Solving \cref{eq:ep_error_eq} is as difficult as solving the original problem. Let $\WW_N^+\subset\WW^+$ be closed, and consider the Galerkin projection $\P_G:\WW^+\to\WW_N^+$ onto $\WW_N^+$ defined by
\begin{align}\label{eq:galerkin_projection}
    \langle (\E-\K) \P_G w,v\rangle = \langle (\E-\K) w,v\rangle\quad\text{for all } v\in\WW_N^+.
\end{align}
Using  \cref{eq:ep_error_eq}, the correction $\ue_{c,n}=\P_G e^+_{n+1/2}$, is then characterized as the solution to
\begin{align}\label{eq:ep_corr}
    \langle (\E-\K)\ue_{c,n} ,v\rangle = \langle(\E-\K)\P_1-\Id)((\E-\K)\ue_n-\ell),v\rangle\quad\text{for all } v\in\WW_N^+,
\end{align}
where the right-hand side involves available data only. The update is performed via
\begin{align}\label{eq:ep_update}
    \ue_{n+1} = \ue_{n+\frac{1}{2}} + \ue_{c,n}.
\end{align}

Since $\P_G$ is non-expansive in the norm induced by $\E-\K$, the error analysis for the overall iteration \cref{eq:richardson_general,eq:ep_update} relies on the spectral properties of $\P_1$. Therefore, the following theoretical investigations consider the generalized eigenvalue problem
\begin{align}\label{eq:richardson_gevp}
    (\E-\K)w = \lambda \P_1^{-1} w.
\end{align}
The following lemma is well-known and we provide a proof for later reference.
\begin{lemma}\label{lem:convergence_richardson_general}
	Let $0<\beta\leq 1$ and assume that the eigenvalues $\lambda$ of \cref{eq:richardson_gevp} satisfy $\beta\leq \lambda\leq 1$. Then, for any $\ue_n\in\WW^+$, $\ue_{n+1/2}$ defined via \cref{eq:richardson_general} satisfies
	\begin{align*}
		\|\ue-\ue_{n+\frac{1}{2}}\|_{\E-\K} \leq (1-\beta) \|\ue-\ue_n\|_{\E-\K}.
	\end{align*}
\end{lemma}
\begin{proof}
Assume that $\{(w_k,\lambda_k)\}_{k\geq 0}$ is the eigensystem of the generalized eigenvalue problem \cref{eq:richardson_gevp}.
For any $\ue_n$, the error $e^+_n=\ue-\ue_n$ satisfies
\begin{align}\label{eq:ep_richardson_error}
e^+_{n+\frac{1}{2}}=(\Id-\P_1(\E-\K))e^+_n.
\end{align} 
Using the expansion $e^+_n=\sum_{k=0}^\infty a_k w_k$, we compute
$\|e^+_n\|^2_{\E-\K} = \sum_{k=0}^\infty a_k^2 \lambda_k$.
Using \cref{eq:ep_richardson_error}, we thus obtain 
$e^+_{n+1/2} = \sum_{k=0}^\infty (1-\lambda_k) a_k w_k$, and hence
\[
\|e^+_{n+\frac{1}{2}}\|_{\E-\K}^2 = \sum_{k=0}^\infty (1-\lambda_k)^2 \lambda_k a_k^2 \leq \sup_{0\leq k<\infty}(1-\lambda_k)^2 \|e^+_n\|^2_{\E-\K}.
\]
Since $0<\beta\leq \lambda_k\leq 1$ by assumption, the assertion follows.
\end{proof}
The next statement is a direct consequence of \cref{lem:convergence_richardson_general} and the observation that $e_{n+1}^+=(\I-\P_G)e_{n+1/2}^+$ satisfies 
\begin{align}\label{eq:bestapprox}
    \|e_{n+1}^+\|_{\E-\K}=\inf_{v\in\WW^+_N} \|e_{n+\frac{1}{2}}^+-v\|_{\E-\K}.
\end{align}
\begin{lemma}\label{lem:convergence_richardson_ep}
	Let $\WW_N^+\subset \WW^+$ be closed, and assume that the eigenvalues $\lambda$ of \cref{eq:richardson_gevp} satisfy $\beta\leq \lambda\leq 1$ for some $0<\beta\leq 1$.
	Then, for any $\ue_0\in\WW^+$, the sequence $\{\ue_n\}$ defined in \cref{eq:richardson_general,eq:ep_update} converges linearly to the solution $\ue$ of \cref{eq:ep}, i.e., 
	\begin{align}
		\|\ue-\ue_{n+1}\|_{\E-\K} \leq (1-\beta)\|\ue-\ue_n\|_{\E-\K}.
	\end{align}
\end{lemma}
In view of the previous lemma fast convergence $\ue_n\to\ue$ can be obtained by ensuring that $\beta$ is close to one or by making the best-approximation error in \cref{eq:bestapprox} small.
These two possibilities are discussed next.

\begin{lemma}\label{lem:spectral_equivalence}
Let $\P_1$ be defined either by (i) $\P_1^{-1}=\E$ or (ii)
\[
\P_1^{-1}=\E_0=(1-cg)^{-1} \A'\M^{-1}\A + \M + \R.
\]
Then $\P_1$ is spectrally equivalent to $\E-\K$, i.e.,
\begin{align*}
 (1-c)(\P_1^{-1} w^+,w^+) \leq ((\E-\K)w^+,w^+) \leq (\P_1^{-1} w^+,w^+),
\end{align*}
for all $w^+\in\WW^+$. It holds $1-\beta=c$ in \cref{lem:convergence_richardson_ep} in both cases.
\end{lemma}
\begin{proof}
Since $\A w^+\in\VV^-$, the result is a direct consequence of \cref{lem:spectral_equivalence_inf}.
\end{proof}

\begin{remark}\label{rem:ep_choices}
We can further generalize the choices for $\P_1^{-1}$ by choosing $N^+\geq -1$, $N^-\geq 0$, and $\gamma_{N^-}=1/(1-cg^{N^-+1})$. Then
$$\P_1^{-1}=\P^{-1}_{N^+,N^-}=\R + \gamma_{N^-}\A' (\M-\K_{N^-})^{-1}\A + \M -\K_{N^+}$$ and $\E-\K$ are spectrally equivalent, i.e.,
\begin{align*}
(1-cg^{\min(N^-,N^+)+1})(\P_1^{-1} w^+,w^+) \leq ((\E-\K)w^+,w^+) \leq (\P_1^{-1} w^+,w^+)
\end{align*}
for all $w^+\in\WW^+$. In particular, $1-\beta=cg^{\min(N^-,N^+)+1}$ in \cref{lem:convergence_richardson_ep}. \end{remark}
\begin{remark} For isotropic scattering $g=0$, we have that $\E=\E_0$. Thus, the above choices can both be understood as generalizations of the iteration considered in \cite{palii2020convergent}.
\end{remark}

The preconditioners of \cref{rem:ep_choices} yield arbitrarily small contraction rates for sufficiently large $N^+$ and $N^-$. However, the efficient implementation of such a preconditioner seems to be rather challenging.
Therefore, we focus on the preconditioners defined in \cref{lem:spectral_equivalence_inf} in the following. Since these choices for $\P_1$ yield slow convergence for $c\approx 1$, we need to construct $\WW_N^+$ properly. This is done along the following lines, see \cref{sec:subspace_correction} for precise definition.

\subsection{A motivation for constructing effective subspaces}\label{sec:motivation}
From the proof of \cref{lem:convergence_richardson_general}, one sees that error modes associated to small eigenvalues $\lambda$ of \cref{eq:richardson_gevp} converge slowly. Hence, in order to regain fast convergence, such modes should be approximated well by functions in $\WW_N^+$, see \cref{eq:bestapprox}. Next we give a heuristic motivation that such slowly convergent modes might be approximated well by low-order spherical harmonics.

Since we use $\P_1^{-1}\approx\E$ below, let us fix $\P_1^{-1}=\E$ in this subsection. Furthermore, let $w$ be a slowly damped mode satisfying \cref{eq:richardson_gevp} with $\lambda\approx 1-c \approx 0$. Then $w$ satisfies
$\K w = \delta \E w$ with $\delta=1-\lambda\approx c \approx 1$. Let us expand the angular part of $w$ into spherical harmonics,
\begin{align*}
	    w(\s,\r) = \sum_{l=0}^\infty \sum_{m=-l}^l w^l_m(\r) H^l_m(\s),
\end{align*}
where $w^l_m=0$ if $l$ is odd. As in the proof of \cref{lem:spectral_equivalence}, we obtain
\begin{align*}
    \K w = \sum_{l=0}^\infty g^l \sum_{m=-l}^l \ss(r) w^l_m(\r) H^l_m(\s).
\end{align*}
Since $\ss\leq \st$, orthogonality of the spherical harmonics implies
\begin{align*}
    \sum_{l=0}^\infty c g^l \sum_{m=-l}^l \|\sqrt{\st} w^l_m\|_{L^2(R)}^2\geq (\K w,w)&=\\
    \delta \bigg( \langle\R w,w\rangle + \|\sgrad &w\|_{(\M-\K)^{-1}}^2 + \sum_{l=0}^\infty \sum_{m=-l}^l \|\sqrt{\st} w^l_m\|_{L^2(R)}^2\bigg).
\end{align*}
Neglecting the contributions from $\R$ and $\sgrad$, we see that
\begin{align}\label{eq:positivity}
    \sum_{l=0}^\infty (c g^l-\delta) \sum_{m=-l}^l \|\sqrt{\st} w^l_m\|_{L^2(R)}^2\geq 0.
\end{align}
Since $\delta\approx c\approx 1$ by assumption and $g<1$, \cref{eq:positivity} can hold true only if $w$ can be approximated well by spherical harmonics of degree less than or equal to $N$ for some moderate integer $N$.

Note that this statement quantifies approximation in terms of the $L^2$-norm. However, using recurrence relations of spherical harmonics to incorporate the terms $\langle \R w,w\rangle+\|\sgrad w\|^2_{(\M-\K)^{-1}}$ into \cref{eq:positivity}, suggests that a similar statement also holds for the $\E-\K$-norm. A full analysis of this statement seems out of the scope of this paper, and we postpone it to future research.
We conclude that effective subspaces $\WW_N^+$ consist of linear combinations of low-order spherical harmonics, and we employ this observation in our numerical realization.

\section{Galerkin approximation}\label{sec:galerkin}
The iterative scheme of the previous section has been formulated for infinite-dimensional function spaces $\WW^+$ and $\WW_N^+\subset \WW^+$.
For the practical implementation we recall the approximation spaces described in \cite{egger2012mixed} and \cite[Section 6.3]{palii2020convergent}. Let $\ThR$ and $\ThS$ denote shape regular triangulations of $R$ and $S$, respectively. For simplicity we assume the triangulations to be quasi-uniform.
To properly define even and odd functions associated with the triangulations, we further require that $-K_S\in \ThS$ for each spherical element $K_S\in\ThS$. The latter requirement can be ensured by starting with a triangulation of a half-sphere and reflection.
Let $\XX_h^+=\PP_1^c(\ThR)$ denote the vector space of continuous, piecewise linear functions subordinate to the triangulation $\ThR$ with basis $\{\varphi_i\}$ and dimension $n_R^+$, and let $\XX_h^-=\PP_0(\ThR)$ denote the vector space of piecewise constant functions subordinate to $\ThR$ with basis $\{\chi_j\}$ and dimension $n_R^-$. Similarly, we denote by $\SS_h^+=\PP_0(\ThS)\cap L^2(S)^+$ and $\SS_h^-=\PP_1(\ThS)\cap L^2(S)^-$ the vector spaces of even, piecewise constant and odd, piecewise linear functions subordinate to the triangulation $\ThS$, respectively.
We can construct a basis $\{\mu_k^+\}$ for $\SS_h^+$ by choosing $n_S^+$ many triangles with midpoints in a given half-sphere, and define the functions $\mu_k^+$ to be the indicator functions of these triangles. For any other point $s\in S$, we find $K_S\in \ThS$ with midpoint in the given half-sphere such that $-s\in K_S$ and we define $\mu_k^+(s)=\mu_k^+(-s)$. A similar construction leads to a basis $\{\psi_l^-\}$ of $\SS_h^-$.
The conforming approximation spaces are then defined through tensor product constructions, $\WW_h^+=\SS_h^+\otimes \XX_h^+$, $\VV_h^-=\SS_h^-\otimes\XX_h^-$. Thus, for some coefficient matrices $\big[\bUe_{i,k}\big]\in\RR^{n_R^+\times n_S^+}$ and $\big[\bUo_{j,l}\big]\in\RR^{n_R^-\times n_S^-}$, any $\ue_h\in\WW_h^+$ and $\uo_h\in\VV_h^-$ can be expanded as
\begin{align}\label{eq:solutionrepresentation}
    \ue_h = \sum_{i=1}^{n_R^+}\sum_{k=1}^{n_S^+} \bUe_{i,k} \varphi_i \mu_k^+,\qquad
    \uo_h = \sum_{j=1}^{n_R^-}\sum_{l=1}^{n_S^-} \bUo_{j,l} \chi_j \psi_l^-.
\end{align}
The Galerkin approximation of \cref{eq:op1}--\cref{eq:op2} computes $(\ue_h,\uo_h)\in\WW_h^+\times\VV_h^-$ such that
\begin{align}
    \R \ue_h - \A' \uo_h + \M \ue_h &= \K \ue_h +\ell^+ \qquad\text{ in } (\WW_h^+)', \label{eq:op1h}\\
    \A \ue_h + \M \uo_h &= \K \uo_h + \ell^- \qquad \text{ in } \VV_h^-.\label{eq:op2h}
\end{align}
The discrete mixed system \cref{eq:op1h}--\cref{eq:op2h} can be solved uniquely \cite{egger2012mixed}.
Denoting $\bu^\pm=\operatorname{vec}(\bU^\pm)$ the concatenation of the columns of the matrices $\bU^\pm$ in a vector, the mixed system \cref{eq:op1h}--\cref{eq:op2h} can be written as the following linear system
\begin{align}\label{eq:mixed_discrete}
    \begin{bmatrix}\bR + \bMe & -\bA^\intercal\\ \bA & \bMo\end{bmatrix}\begin{bmatrix}\bue\\ \buo\end{bmatrix} = \begin{bmatrix} \bKe&\\ &\bKo \end{bmatrix}\begin{bmatrix}\bue\\ \buo\end{bmatrix} + \begin{bmatrix}\ble\\ \blo\end{bmatrix}.
\end{align}
The matrices in the system are given by
\begin{align}
\bKe={}&\bSe\otimes\bMse_s,&
\bKo={}&\bSo\otimes\bMso_s,\label{eq:bKebKo}\\
\bMe={}&\bMSpheree\otimes\bMse_t,&
\bMo={}&\bMSphereo\otimes\bMso_t,\label{eq:bMebMo}\\
\bA={}&\sum_{i=1}^d\bAi_i\otimes \bD_i,&
\bR={}&\operatorname{blkdiag}(\bRs_1,\ldots,\bRs_{n_S^+}),\label{eq:AR}
\end{align}
where we denote by Gothic letters the matrices arising from the discretization on $R$ and by Sans Serif letters matrices arising from the discretization on $S$, i.e.,
\begin{align*}
    (\bMso_t)_{j,j'} &=\int_R \sigma_t \chi_j \chi_{j'} dr,						& (\bSo)_{l,l'}    &= \int_S \S \psi_l^-\psi_{l'}^-ds,\\
    (\bMse_t)_{i,i'} &=\int_R \sigma_t \varphi_i \varphi_{i'} dr,				& (\bSe)_{k,k'}    &= \int_S \S \mu_k^+\mu_{k'}^+ ds,\\
    (\bD_n)_{j,i}    &= \int_R \frac{\partial \varphi_i}{\partial r_n} \chi_j dr,& (\bAi_n)_{l,k}   &= \int_S s_n \psi_{l}^- \mu_k^+ ds,\\
    (\bRs_k)_{i,i'}  &= \int_{\partial R} \varphi_i \varphi_{i'} \omega_k dr,	 & \omega_k&=\int_{S}|\sn|(\mu_k^+)^2ds .
\end{align*}
The matrices $\bMso_s$ and $\bMse_s$ are defined accordingly.
By $\bMSpheree$ and $\bMSphereo$ we denote the Gramian matrices in $L^2(S)$. 
We readily remark that all of these matrices are sparse, except for $\bSe$ and $\bSo$ which are dense. $\bMSpheree$ and $\bMSphereo$ are diagonal and $3\times3$ block diagonal, respectively.
Moreover, we note that $\bMso_t$ is a diagonal matrix.

To conclude this section let us remark that taking the Schur complement of \cref{eq:mixed_discrete} finally yields the matrix counterpart of the even-parity system \cref{eq:ep}, i.e.,
\begin{align}\label{eq:ep_discrete}
    \bE \bue = \bKe \bue + \bq
\end{align}
with $\bE=\bA\!^{\intercal}(\bMo-\bKo)^{-1}\bA+\bMe+\bR$ and $\bq=\ble+\bA\!^{\intercal} (\bMo-\bKo)^{-1}\blo$.

\section{Discrete preconditioned Richardson iteration}\label{sec:richardson_discrete}
After discretization, the iteration presented in \cref{sec:iteration_ep} becomes
\begin{align}\label{eq:Richardson_discrete}
    \bue_{n+1}=\bue_n-\bP_2\bP_1((\bE-\bKe)\bue_n-\bq).
\end{align}
The preconditioner $\bP_1$ is directly related to $\P_1$ in \cref{eq:richardson_general}. By denoting the coordinate vectors of the basis functions of the subspace $\WW^+_{h,N}\subset\WW^+_h$ by $\bW$, the matrix representation of the overall preconditioner is
\begin{align}\label{eq:product_pre}
\bP_2\bP_1 = \bP_1+ \bW\big(\bW\!^{\intercal}(\bE-\bKe)\bW\big)^{-1}\bW\!^{\intercal}(\bIde-(\bE-\bKe)\bP_1).
\end{align}
Denoting $\bP_G=\bW\big(\bW\!^{\intercal}(\bE-\bKe)\bW\big)^{-1}\bW\!^{\intercal}(\bE-\bKe)$ the matrix representation of the Galerkin projection $\P_G$ defined in \cref{eq:galerkin_projection}, the iteration matrix admits the factorization
\[
\bIde-\bP_2\bP_1(\bE-\bKe) = (\bIde-\bP_G)\big(\bIde-\bP_1(\bE-\bKe)\big).
\]
The discrete analog of \cref{lem:convergence_richardson_ep} implies that the sequence $\{\bue_n\}$ generated by \eqref{eq:Richardson_discrete} converges for any initial choice $\bue_0$ to the solution $\bue$ of \cref{eq:ep_discrete}. More precisely, by choosing $\bP_1$ according to \cref{lem:spectral_equivalence}, there holds
\begin{align}\label{eq:convergence_discrete}
    \|\bue-\bue_{n+1}\|_{\bE-\bKe}\leq \eta \|\bue-\bue_{n}\|_{\bE-\bKe},
\end{align}
where $0\leq \eta\leq c<1$ is defined as
\begin{align}\label{eq:def_eta}
   \eta=\sup \|(\bIde-\bP_G)(\bIde-\bP_1(\bE-\bKe))\bv^+\|_{\bE-\bKe}
\end{align}
with supremum taken over all $\bv^+\in\RR^{n_S^+ n_R^+}$ satisfying $\|\bv^+\|_{\bE-\bKe}=1$.
The realization of \cref{eq:Richardson_discrete} relies on the efficient application of $\bE$, $\bKe$, $\bP_1$ and $\bP_2$ discussed next.
\subsection{Application of $\bE$}\label{sec:pre_M-K}
In view of \cref{eq:bMebMo,eq:AR} it is clear that $\bA$, $\bMe$, and $\bMo$ can be stored and applied efficiently by using their tensor product structure, sparsity, and the characterization
\begin{align}\label{eq:tpmv}
(\bB\otimes\bC)\operatorname{vec}(\bX)=\operatorname{vec}(\mathbf{D})
\quad\Longleftrightarrow\quad
\bC\bX\bB^\intercal=\mathbf{D},
\end{align}
where $\bC\in\RR^{m\times n}$, $\bX\in\RR^{n\times p}$, $\bB\in\RR^{q\times p}$, $\mathbf{D}\in\RR^{m\times q}$. The boundary matrix $\bR$ consists of sparse diagonal blocks, and can thus also be applied efficiently, see \cref{sec:complexity} for details. The remaining operation required for the application of $\bE$ as given in \cref{eq:ep_discrete} is the application of $(\bMo-\bKo)^{-1}$, which deserves some discussion.
Since $\bMo-\bKo$ has a condition number of $(1-cg)^{-1}$ due to \cref{lem:spectral_equivalence_inf}, a straightforward implementation with the conjugate gradient method may be inefficient for $cg\approx 1$.
To mitigate the influence of $cg$, we can use \cref{lem:spectral_equivalence_inf} once more and obtain preconditioners derived from $\M-\K_N$, which lead to bounds on the condition number by $(1-(cg)^{N+2})^{-1}$ for odd $N$. In what follows, we comment on the practical realization of such preconditioners and their numerical construction. As we will verify in the numerical examples, these preconditioners allow the application of $(\bMo-\bKo)^{-1}$ in only a few iterations even for $g$ close to $1$.

After discretization, the continuous eigenvalue problem \cref{eq:eigS} for the scattering operator becomes the generalized eigenvalue problem
\[
\bSo \bWSphereo=\bMSphereo\bWSphereo\bLambda\!^-.
\]
Since $\bSo$ and $\bMSphereo$ are symmetric and positive, the eigenvalues satisfy $0\leq\lambda_l\leq g$, and we assume that they are ordered non-increasingly. The eigenvectors $\bWSphereo$ form an orthonormal basis $(\bWSphereo)\!^{\intercal}\bMSphereo\bWSphereo=\bIdSphereo$. Truncation of the eigendecomposition at index $d_N=(N+1)(N+2)/2$, $N$ odd, which is the number of odd spherical harmonics of order less than or equal to $N$, yields the approximation
\begin{align}\label{eq:lowrankscatteringmatrix}
\bSo = \bMSphereo\bWSphereo\bLambda\!^-(\bWSphereo)\!^{\intercal}\bMSphereo \approx \bMSphereo\bWSphereo_N\bLambda\!^-_N(\bWSphereo_N)^\intercal\bMSphereo=:\bSo_N.
\end{align}
The discrete version of $\M-\K_N$ then reads $\bMo-\bKo_N$, with $\bKo_N=\bSo_N\otimes \bMso_s$. 
An explicit representation of its inverse is given by the following lemma.
Its essential idea is to use an orthogonal decomposition of $\VV_h^-$ induced by the eigendecomposition of $\bSo$, and to employ the diagonal representation of $\bMo-\bKo_N$ in the angular eigenbasis.
\begin{lemma}\label{lem:apply_high_order_pre}
Let $\mathbf{b}\in\RR^{n_S^-n_R^-}$. Then $\bx=(\bMo-\bKo_N)^{-1}\bb$ is given by
\begin{align}\label{eq:MmKprecond}
\begin{aligned}
\bx
={}&
\Big(\bWSphereo_N\otimes\bIdso\Big)\Big(\bIdSphereo\otimes\bMso_t-\bLambda\!^-_N\otimes\bMso_s\Big)^{-1}\Big((\bWSphereo_N)^\intercal\otimes\bIdso\Big)\mathbf{b}\\
&\hspace*{2.06cm}+
\Big(\Big((\bMSphereo)^{-1}-\bWSphereo_N(\bWSphereo_N)^\intercal\Big)\otimes(\bMso_t)^{-1}\Big)\mathbf{b},
\end{aligned}
\end{align}
where $\bIdso$ and $\bIdSphereo$ denote the identity matrices of dimension $n_R^-$ and $d_N$, respectively.
\end{lemma}
\begin{proof}
We first decompose $\bx$ as follows
\begin{align}\label{eq:splitx}
\bx
=
(\bWSphereo_N(\bWSphereo_N)\!^{\intercal}\bMSphereo\otimes\bIdso)\mathbf{x}
+
\big((\bIdSphereo-\bWSphereo_N(\bWSphereo_N)^\intercal\bMSphereo)\otimes\bIdso\big)\mathbf{x}.
\end{align}
Applying $(\bWSphereo_N)^\intercal\otimes\bIdso$ to $(\bMo-\bKo_N)\bx=\bb$, \cref{eq:lowrankscatteringmatrix}, and $\bMSphereo$-orthogonality of $\bWSphereo_N$ yield
\[
\big(\bIdSphereo\otimes\bMso_t-\bLambda\!^-_N\otimes\bMso_s\big)((\bWSphereo_N)^\intercal\bMSphereo\otimes\bIdso)\mathbf{x}=((\bWSphereo_N)^\intercal\otimes\bIdso)\mathbf{b}.
\]
Inverting $\bIdSphereo\otimes\bMso_t-\bLambda\!^-_N\otimes\bMso_s$ and applying $\bWSphereo_N\otimes\bIdso$ further yields
\[
(\bWSphereo_N(\bWSphereo_N)^\intercal\bMSphereo\otimes\bIdso)\mathbf{x}
=
(\bWSphereo_N\otimes\bIdso)\big(\bIdSphereo\otimes\bMso_t-\bLambda\!^-_N\otimes\bMso_s\big)^{-1}((\bWSphereo_N)^\intercal\otimes\bIdso)\mathbf{b}.
\]
For the other part in \cref{eq:splitx}, apply $((\bMSphereo)^{-1}-\bWSphereo_N(\bWSphereo_N)^\intercal)\otimes(\bMso_t)^{-1}$ to $(\bMo-\bKo_N)\bx=\bb$ and obtain
\[
((\bIdSphereo-\bWSphereo_N(\bWSphereo_N)^\intercal\bMSphereo)\otimes\bIdso)\mathbf{x}
=
\big(((\bMSphereo)^{-1}-\bWSphereo_N(\bWSphereo_N)^\intercal)\otimes(\bMso_t)^{-1}\big)\mathbf{b}.
\]
Substituting both expressions into \cref{eq:splitx} yields the assertion.
\end{proof}

\begin{remark}
If $\ss$ has huge variations, a more effective approximation to $\bKo$ can be obtained from the eigendecomposition
\[ \bMso_s \bIdso = \bMso_t \bIdso \Delta\]
with diagonal matrix $\Delta$ with entries $\Delta_{j}=\int_R \ss \chi_jdr/\int_R\st\chi_j dr$. 
The modified approximation $\widetilde\bKo$ is then computed by considering only those combinations of spatial and angular eigenfunctions for which $\lambda_l\Delta_j$ is above a certain tolerance.
\end{remark}

\subsection{Application of $\bKe$ and $\bKo$}
Although $\bKe$ and $\bKo$ provide a tensor product structure \cref{eq:bKebKo} involving the sparse matrices $\bMse_s$ and $\bMso_s$, the density of the scattering operators $\bSe$ and $\bSo$ becomes a bottleneck for iterative methods due to quadratic complexity in storage consumption and computational cost for assembly and matrix-vector products. $\mathcal{H}$- and $\mathcal{H}^2$-matrices, which can be considered as abstract variants of the fast multipole method \cite{FD2009,GR1987}, where developed in the context of the boundary element method and can realize the storage, assembly and matrix-vector multiplication in linear or almost linear complexity, see \cite{Bor2010,Hac2015} and the references therein. A sufficient condition for compressibility in these formats is the following.
\begin{definition}\label{def:prel.asSmooth}
Let $\tilde{S}\subset\mathbb{R}^d$ such that
\(k\colon \tilde{S}\times \tilde{S}\to\mathbb{R}\) is defined
and arbitrarily often differentiable for all $\tilde{\bfx}\neq\tilde{\bfy}$ with
$\tilde{\bfx},\tilde{\bfy}\in\tilde{S}$.
\(k\) is called \emph{asymptotically smooth} if 
\begin{align}
\big|\partial_{\tilde{\bfx}}^{\boldsymbol{\alpha}}\partial_{\tilde{\bfy}}^{\boldsymbol\beta}
k(\tilde{\bfx},\tilde{\bfy})\big|
\leq
C\frac{(|\boldsymbol{\alpha}|+|\boldsymbol{\beta}|)!}
{r^{{|\boldsymbol\alpha|+|\boldsymbol\beta|}}}
\|\tilde{\bfx}-\tilde{\bfy}\|^{-|\boldsymbol\alpha|-|\boldsymbol\beta|},\qquad\tilde{\bfx}\neq\tilde{\bfy},
\end{align}
independently of \(\boldsymbol\alpha\) and \(\boldsymbol\beta\) for some constants \(C,r>0\).
\end{definition}
While several methods \cite{DHS2006,DHP2016} can operate on the Henyey-Greenstein kernel on the sphere, most classical methods require an extension into space which we define as
\begin{align}\label{eq:HGextension}
	K(\tilde{\bfx},\tilde{\bfy}) = k(\bfx\cdot\bfy),
	\qquad
	\text{with}~\bfx=\tilde{\bfx}/\|\tilde{\bfx}\|,~\bfy=\tilde{\bfy}/\|\tilde{\bfy}\|.
\end{align}
The following result allows to use this extension in most $\mathcal{H}$- and $\mathcal{H}^2$-matrix libraries such as \cite{Boe,DHK+2020,Kri} in a black-box fashion.
\begin{lemma}
Let $g\geq 0$. Then  $K(\tilde{\bfx},\tilde{\bfy})$ is asymptotically smooth for $\tilde{\bfx},\tilde{\bfy}\in\RR^d\setminus \{0\}$.
\end{lemma}
\begin{proof}
We first remark that the cosinus theorem implies for $\bfx,\bfy\in S$ with angle $\varphi$ that $\bfx\cdot\bfy = \cos(\varphi) = 1-\|\bfx-\bfy\|^2/2$. Moreover, $\tilde{k}(\xi) = k(1-\xi^2/2)$ is holomorphic for $\Re(\xi)>0$ such that its Taylor series around $\xi>0$ has convergence radius $\xi$ and the derivatives of $\tilde{k}$ satisfy $\big|\partial_\xi^\alpha\tilde{k}(\xi)\big|\leq cr^\alpha\alpha!|\xi|^{-\alpha}$, $\alpha\in\NN_0$, for all $\xi>0$. Since $\tilde{\bfx}\mapsto\bfx=\tilde{\bfx}/\|\tilde{\bfx}\|$ is analytic for $\tilde{\bfx}\neq 0$ and since $K(\tilde{\bfx},\tilde{\bfy})=\tilde{k}(\|\bfx-\bfy\|)$, the assertion follows in complete analogy to the appendix of \cite{HP2013}.
\end{proof}
The $\mathcal{H}$- or $\mathcal{H}^2$-approximation of $\bSe$ and $\bSo$ and the sparsity of $\bMse_s$ and $\bMso_s$ combined with the tensor product identity \cref{eq:tpmv} then allow for an application of $\bKe$ and $\bKo$ in almost linear or even linear complexity.

\subsection{Choice and implementation of $\bP_1$}\label{sec:pre_p1}
As shown in \cref{sec:iteration_ep}, choosing $\bP_1$ as in \cref{lem:spectral_equivalence} leads to contraction rates $\eta\leq c$ in \eqref{eq:convergence_discrete}, i.e., independent of the mesh-parameters.
The choice $\bP_1=\bE^{-1}$ can be realized through an inner iterative methods, such as a preconditioned Richardson iteration resulting in an inner-outer iteration scheme when employed in \cref{eq:Richardson_discrete}. 
An effective preconditioner for $\bE$ is given by the block-diagonal, symmetric positive definite matrix $\bE_0=\frac{1}{1-cg}\bA^\intercal(\bM^-)^{-1}\bA+\bR+\bM^+$ which provides the spectral estimates
\begin{align}\label{eq:equivalence_bE_bE0}
    (1-cg)\bx^\intercal\bE_0\bx\leq \bx^\intercal\bE\bx\leq \bx^\intercal\bE_0\bx,
\end{align}
for all $\bx\in\RR^{n_S^+n_R^+}$, cf. \cref{lem:spectral_equivalence_inf}. Thus, the condition number of $\bE_0^{-1}\bE$ is bounded by $(1-cg)^{-1}$, which is uniformly bounded for $c\in[0,1]$ for fixed $g<1$.
For clarity of presentation, we will use a preconditioned Richardson iteration for the inner iteration to implement $\bP_1$ in the rest of the paper, but remark that a non-stationary preconditioned conjugate gradient method will lead to even better performance.
Applying $\bP_1$ with high accuracy may still involve many iterations. Instead, we use a preconditioner $\bP_1^l$ which performs $l$ steps of the inner iteration, i.e., we set $\bP_1^l \bb=\bz_l$, where
\begin{align}\label{eq:applyP1}
    \bz_0=0,\qquad \bz_{k+1}=\bz_{k}-\bE_0^{-1}(\bE \bz_k-\bb),\quad k<l. 
\end{align}
Notice that, $\bP_1^1=\bE_0^{-1}$ while $\bP_1^lb \to \bE^{-1}\bb$ as $l\to\infty$.
The next result asserts that this inexact realization of the preconditioner leads to a convergent scheme.
\begin{lemma}\label{lem:preconditioner_pl}
Let $l\geq 1$ be fixed. The iteration \cref{eq:ep_discrete} with preconditioner $\bP_1=\bP_1^l$ defines a convergent sequence, i.e., \cref{eq:convergence_discrete} holds with $\eta\leq c$ and $\eta$ as in \cref{eq:def_eta}.
\end{lemma}
\begin{proof}
Observing that $\bP_1^l = \sum_{k=0}^{l-1} (\bE_0^{-1}(\bE_0-\bE))^k\bE_0^{-1}$ and that each term in the sum is symmetric and positive semi-definite for $k>0$ and positive definite for $k=0$, it follows that $\bP_1^l$ is symmetric positive definite. Using \cref{eq:equivalence_bE_bE0}, we deduce that the sum converges as a Neumann series to $\bE^{-1}$. Hence, it follows that for all $\bx\in\RR^{n_S^+n_R^+}$
\begin{align*}
    \bx^{\!\intercal} \bE_0^{-1}\bx\leq \bx^{\!\intercal} \bP_1^{l}\bx\leq \bx^{\!\intercal} \bE^{-1}\bx,
\end{align*}
which implies that
$\bx^{\!\intercal} \bE\bx \leq\bx^{\!\intercal} (\bP_1^{l})^{-1}\bx \leq \bx^{\!\intercal} \bE_0\bx$
and, in turn,
\begin{align*}
    (1-c)\bx^{\!\intercal} (\bP_1^{l})^{-1}\bx \leq  \bx^{\!\intercal} (\bE-\bK)\bx\leq \bx^{\!\intercal} (\bP_1^{l})^{-1}\bx,
\end{align*}
where we used \cref{lem:spectral_equivalence}.
The assertion follows then as in \cref{sec:iteration_ep}.
\end{proof}
\subsection{Implementation of the subspace correction}\label{sec:subspace_correction}
The optimal subspaces for the correction \cref{eq:ep_corr} are constructed from the eigenfunctions associated with the largest eigenvalues of the generalized eigenproblem \cref{eq:richardson_gevp} as can be seen from the proof of \cref{lem:convergence_richardson_ep} . The iterative computation of these eigenfunctions is, however, computationally expensive. Instead, we employ a different, computationally efficient tensor product construction that employs discrete counterparts of low-order spherical harmonics expansions motivated in \cref{sec:motivation}.
More precisely, the subspace for the correction is defined as $\WW_{h,N}^+=\PP_{0,N}(\ThS)\otimes \PP_1^c(\ThR)$, where $\PP_{0,N}(\ThS)\subset \PP_{0}(\ThS)$ is the space spanned by the eigenfunctions  associated to the $d_N=(N+1)(N+2)/2$ largest eigenvalues of the generalized eigenvalue problem
\[
\bSe \bWSpheree = \bMSpheree\bWSpheree\bLambda^+
\]
for the scattering operator, mimicking \cref{eq:eigS} after discretization. 
Note that $d_N$ with $N$ even is the number of even spherical harmonics of order less than or equal to $N$, and $\PP_{0,N}(\ThS)$ approximates their span. 
Denote $\bWSpheree_N$ the corresponding matrix of coefficient vectors.
The  subspace $\WW^+_{h,N}$ is spanned by the columns of the matrix
$\bWe=\bWSpheree_N\otimes \bIdse$.
At the discrete level, the correction equation \cref{eq:ep_corr}, thus, reads as
\begin{align}\label{eq:corr_discrete}
\big({\bWe}\!^{\intercal}(\bE-\bKe)\bWe\big) \bu_c={\bWe}\!^{\intercal}((\bE-\bKe)\bP_1-\bI)((\bE-\bKe)\bu_n-\bq).
\end{align}
The efficient assembly of the matrix on the left-hand side relies on the tensor product structure of $\bKe$ and the choice of $\bWSpheree_N$ as outlined in the following. A simple and direct representation of the scattering operator on $\WW^+_{h,N}$ is obtained by
\[
{\bWe}\!^{\intercal}\bKe\bWe = \bLambda_N^+\otimes \bMse_s.
\]
Similarly, we have that ${\bWe}\!^{\intercal}\bMe\bWe=\bIdSpheree \otimes\bMse_t$, and the block-diagonal structure of $\bR$ allows to compute ${\bWe}\!^{\intercal}\bR\bWe$, i.e. the $(i,j)$th block-entry is given by
\[
 \sum_{k=1}^{n_S^+} \bRs_k (\bWSpheree_N(k,i)\bWSpheree_N(k,j))
\]
which requires $O(n_S^+(n_R^+)^{(d-1)/d}d_N)$ many multiplications. The efficient assembly of the remaining term ${\bWe}\!^{\intercal}\bA\!^{\intercal}(\bMo-\bKo)^{-1}\bA\bWe$ relies on another eigenvalue decomposition which diagonalizes $\bMo-\bKo$ on the column range of $\bA\bWe$. The arguments are similar to those in \cref{sec:pre_M-K} and we leave the details to the reader.

\subsection{Full algorithm and complexity}\label{sec:complexity}
For the convenience of the reader we provide here the full algorithm of our numerical scheme. To simplify presentation we start with the application of $\bE$ as given in \cref{alg:applyE} and the application of $\bP_1$ as given in \cref{alg:applyP1}. The full preconditioned Richardson iteration \cref{eq:Richardson_discrete} is outlined in \cref{alg:iteration}.
\begin{algorithm}
\caption{Apply $\bE$, given a factorization of $\bSo_N$ as in \cref{eq:lowrankscatteringmatrix}.\label{alg:applyE}}
\begin{algorithmic}[1]
\Function{$\bfy=$Apply$\bE$}{$\bx$}
\State Solve $(\bMo-\bKo)\bz=\bA\bx$ with PCG, preconditioned by $(\bMo-\bKo_N)^{-1}$ as in \cref{eq:MmKprecond}
\State $\by=\bA\!^{\intercal}\bz+\bMe\bx+\bR\bx$
\EndFunction
\end{algorithmic}
\end{algorithm}
\begin{algorithm}
\caption{Apply $\bP_1=\bP_1^l$ as given in \cref{eq:applyP1}.\label{alg:applyP1}}
\begin{algorithmic}[1]
\Function{$\bz=$Apply$\bP_1$}{$\bx$}
\State $\bz=0$
\For{$k=0,1,\ldots,l$}
\State $\bz=\bz-\bE_0^{-1}($\Call{Apply$\bE$}{$\bz$}$-\bx)$
\EndFor
\EndFunction
\end{algorithmic}
\end{algorithm}
\begin{algorithm}
\caption{Solve $\bE \bue = \bKe \bue + \bq$\label{alg:iteration} according to \cref{eq:Richardson_discrete}}
\begin{algorithmic}[1]
\State Compute $\bSe_N=\bMSpheree\bWSpheree_N\bLambda\!^+_N(\bWSpheree_N)^\intercal\bMSpheree$
\State Compute $\bSo_N=\bMSphereo\bWSphereo_N\bLambda\!^-_N(\bWSphereo_N)^\intercal\bMSphereo$
\State
\State Compute $\bE_c={\bWe}\!^{\intercal}(\bE-\bKe)\bWe$ as in \cref{sec:subspace_correction}
\State
\State Choose $\bue_0$
\For{$n=0,1,2,\ldots$}
\State
\State $\br=$\Call{Apply$\bE$}{$\bue_n$}$-\bKe\bue_n-\bq$
\State $\bs=$\Call{Apply$\bP_1$}{$\br$}
\State $\bue_{n+1/2}=\bue_n-\bs$
\Comment{Half-step}
\State
\State $\bq_c={\bWe}\!^{\intercal}\big($\Call{Apply$\bE$}{$\bs$}$-\bKe\bs-\bq-\br\big)$
\State Solve $\bE_c\bue_{n+1/2,c}=\bq_c$
\State $\bue_{n+1}=\bue_{n+1/2}+\bWe\bue_{n+1/2,c}$
\Comment{Subspace correction}
\EndFor
\end{algorithmic}
\end{algorithm}

For the efficient implementation of these algorithms one may exploit that, except for $\bR$, all matrices
provide a tensor product structure, see \cref{eq:bKebKo,eq:bMebMo,eq:AR}, allowing for efficient storage in $\cO(n_S^{\pm}+n_R^{\pm})$ or $\cO(c_{\H}n_S^{\pm}+n_R^{\pm})$ complexity by using their sparsity or their $\mathcal{H}^2$-matrix representation\footnote{The storage requirements of $\bKe$ and $\bKo$ are $\cO(c_{\H}n_S^{\pm}\log(n_S^{\pm})+n_R^{\pm})$ if $\mathcal{H}$-matrices are used instead of $\mathcal{H}^2$-matrices. In practice, $c_{\H}$ may depend on additional implementation dependent parameters, see \cite{Bor2010,Hac2015}, which we neglect here for sake of simplicity.}. Here, $c_{\H}$ is a constant related to the compression pattern of the $\H^2$-matrix. The storage requirements and application of $\bR$ have complexity $\cO(n_S^+(n_R^+)^{(d-1)/d})$. The relation \cref{eq:tpmv} then allows for an efficient application of all matrices occurring in \cref{eq:mixed_discrete} in $\cO(n_S^{\pm}n_R^{\pm})$ or $\cO(c_{\H}n_S^{\pm}n_R^{\pm})$ operations. Since the solution vector itself has size $n_S^+n_R^+$, see also \cref{eq:solutionrepresentation}, and since $3n_S^+=n_S^-$ and $n_R^+\sim n_R^-$, all matrices appearing in \cref{eq:mixed_discrete} can be stored and applied with linear complexity.

In the following we elaborate the algorithmic complexities of \cref{alg:applyE,alg:applyP1,alg:iteration} in more detail.

\subsubsection{Application of $\bE$}
The listing of \cref{alg:applyE} directly indicates that the main effort of applying $\bE$ lies in the preconditioned conjugate gradient method for applying $(\bMo-\bKo)^{-1}$. From \cref{lem:apply_high_order_pre}, we obtain that $(\bMo-\bKo_N)^{-1}(\bMo-\bKo)$ is applicable in $\cO((d_N+c_{\H}) n_S^- n_R^-)$ operations, while its condition number is $(1-(cg)^{N+2})^{-1}$. This implies an iteration count for the application of $(\bMo-\bKo)^{-1}$ proportional to $(1-(cg)^{N+2})^{-1/2}$ for $cg\approx 1$ when using the preconditioned conjugate gradient method with a fixed tolerance. The overall complexity for applying $(\bMo-\bKo)^{-1}$ and, thus, also $\bE$ is then $\cO((d_N+c_{\H}) n_S^- n_R^-/(1-(cg)^{N+2})^{1/2})$. We note that typically $d_N\ll c_{\H}$ for moderate $N$.

\subsubsection{Application of $\bP_1$}
$\bP_1^l$ consists of $l-1$ applications of $\bE$ and $l$ applications of $\bE_0^{-1}$. 
Since $\bE_0$ is block-diagonal with $n_S^+$ sparse blocks of size $n_R^+\times n_R^+$, the application of $\bE_0^{-1}$ can be performed in
$\cO(n_S^+ (n_R^+)^\gamma)$
if the inversion of each block has $\cO((n_R^+)^\gamma)$ complexity.
This amounts to
$\cO(l (d_N+c_{\H}) n_S^+ n_R^+/(1-(cg)^{N+2})^{1/2}+ l n_S^+ (n_R^+)^\gamma)$
complexity for the application of $\bP_1^l$. 
For moderate $N$, the subspace correction amounts to solving an elliptic system that is reminiscent of an order $N$ spherical harmonics approximation, which can be solved efficiently with a conjugate gradient method preconditioned by a V-cycle geometric multigrid with Gauss-Seidel smoother, cf. \cite{ArridgeEggerSchlottbom13}.

Let us also remark that each diagonal block of $\bE_0$ discretizes an anisotropic diffusion problem with a diffusion tensor $\st^{-1}\int_{K_S} s\cdot s^\intercal ds$ for $K_S\in\T_h^S$. The results reported in \cite{Hemker84} indicate that such problems can be treated efficiently by multigrid methods with line smoothing allowing for $\gamma=1$. A full analysis in the present context is out of the scope of this paper, but any method that gives $\gamma=1$ allows to perform one step in the Richardson iteration \cref{eq:Richardson_discrete} in linear complexity in the dimension of the solution vector. Although $\gamma>1$, sparse direct solvers may work well, too,  cf. \cref{tab:benchmark}.

\subsubsection{Overall iteration}
We start our considerations by remarking that the truncated eigendecompositions of the smaller matrices $\bSe$ and $\bSo$ can be obtained by a few iterations of an iterative eigensolver. Once this is achieved, the computation of the reduced matrix $\bE_c$ can be achieved in $O(n_S^+n_R^+d_N)$ operations, see \cref{sec:subspace_correction}. Thus, the offline cost for the construction of the preconditioners are $O(n_S^+n_R^+d_N)$. The discussion on the application of $\bE$ and $\bP_1$ shows that a single iteration of \cref{alg:iteration} can be accomplished in $\cO(l (d_N+c_{\H}) n_S^+ n_R^+/(1-(cg)^{N+2})^{1/2}+ l n_S^+ (n_R^+)^\gamma)$ operations.

Let us remark that in the case $\gamma=1$ each iteration has linear complexity and it can be implemented such that it offers a perfect parallel weak scaling in $n_S^+n_R^+$ as long as the number of processors is bounded by $n_S^+$ and $n_R^+$. To see this, we note that, with $\bR$ being the only exception, we are only relying on matrix-vector products of matrices having tensor-product structure (or sums thereof). Using the identity \cref{eq:tpmv}, it is clear that these operations offer the promised weak scaling when these matrix-matrix products are accelerated by a parallelization over the rows and columns of the middle matrix. The matrix $\bR$ does not directly provide such a structure, but its block diagonal structure, cf.~\cref{eq:AR}, provides possibilities for a perfectly weakly scaling implementation as well.

In summary, each step in \cref{eq:Richardson_discrete} can be executed very efficiently with straight-forward parallelization.
In the next section we show numerically that the number of iterations required to decrease the error below a given threshold is small already for small values of $l$ and $N$.

\section{Numerical realization and examples}\label{sec:numerics}
We present the performance of the proposed iterative schemes using a lattice type problem \cite{Brunner05}, see \cref{fig:checkerboard}. Here, $R=(0,7)\times(0,7)$ and $c= \|\ss/\st\|_\infty \approx 0.999$. 
The coarsest triangulation of the sphere consists of $128$ element, i.e., $n_S^+=64$, and $n_R^+=3\,249$ vertices to discretize the spatial domain. Finer meshes are obtained by uniform refinement; the new grid points for $\T_h^S$ are projected to the sphere. To minimize consistency errors, we use higher-order integration rules for the spherical integrals.
\begin{figure}[h!]
	\includegraphics[width=0.49\textwidth]{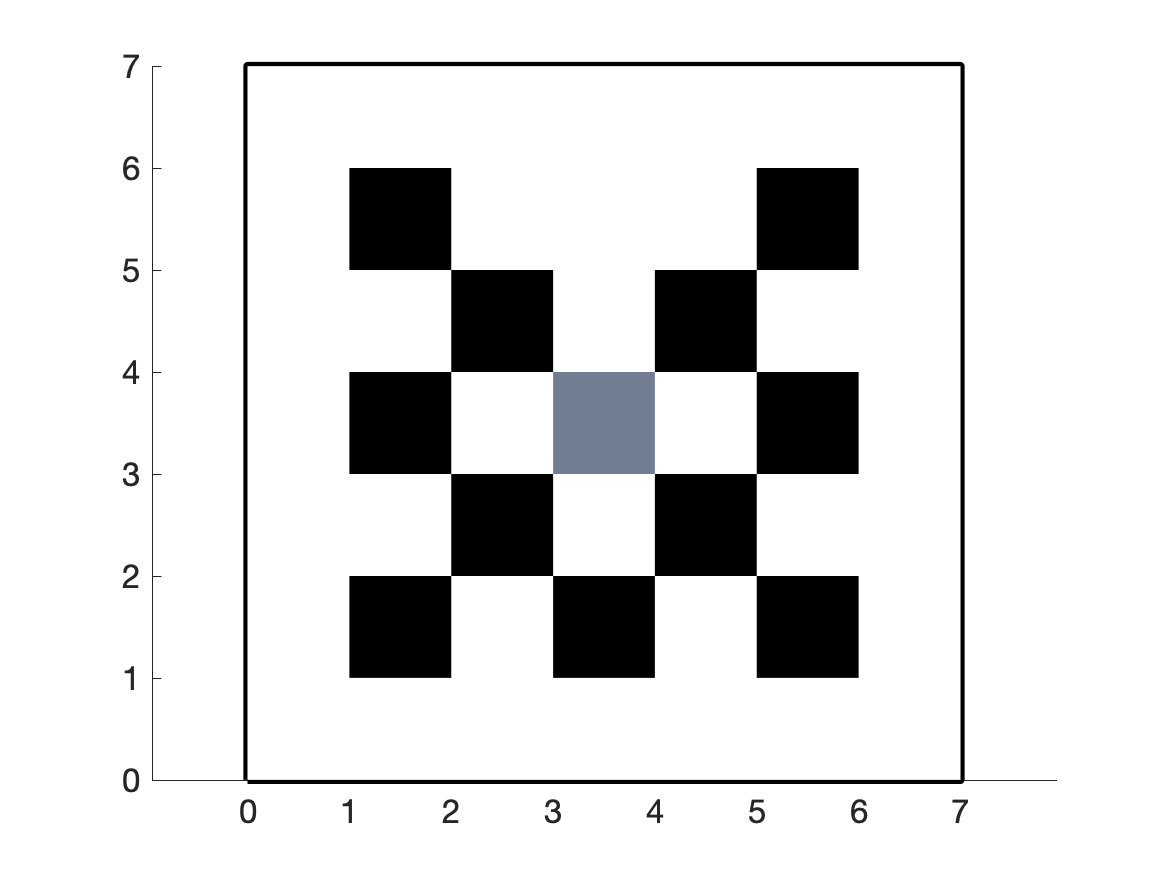}
	\includegraphics[width=0.49\textwidth]{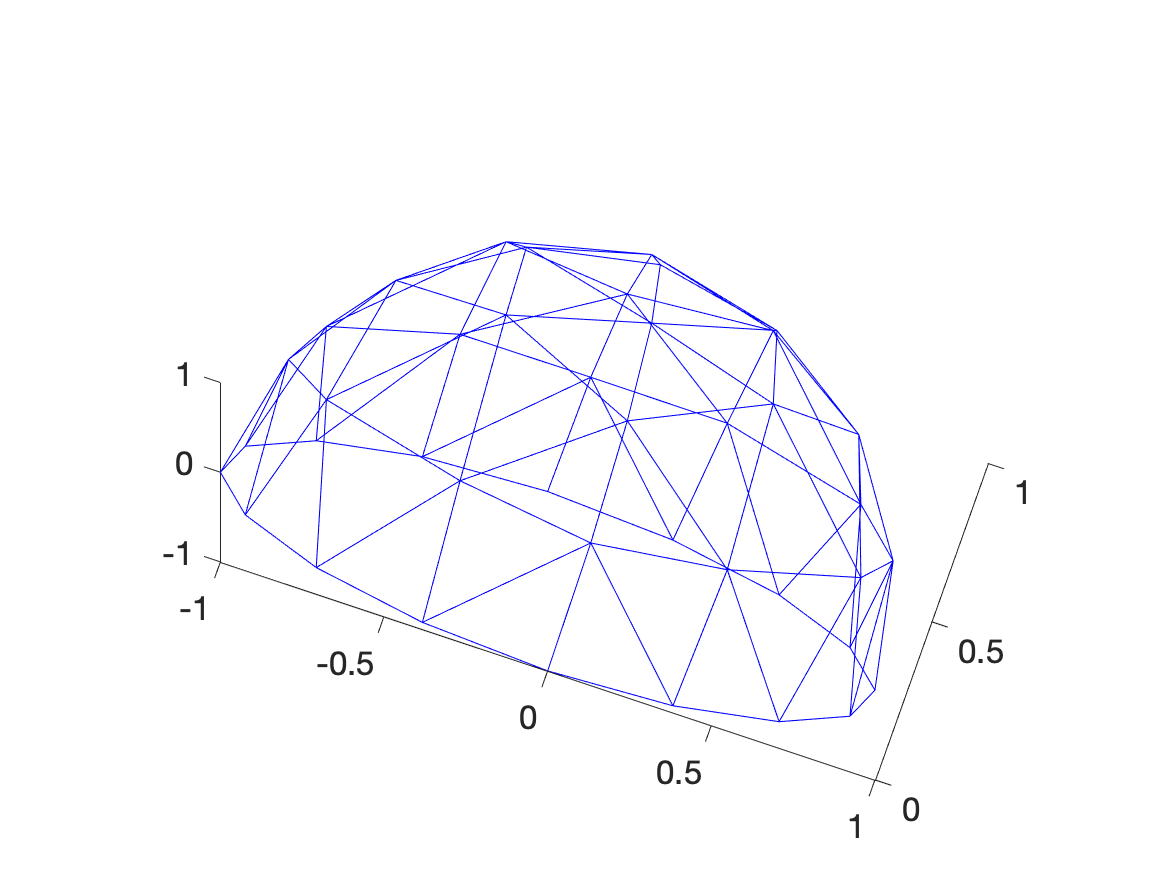}
	\caption{Left: geometry of the lattice problem. The optical parameters are $\ss=10$ and $\sa=0.01$ in the white and grey regions, $\ss=0$ and $\sa=1$ in the black regions and $q=1$ in the grey region and $q=0$ outside the grey region. Right: Sketch of the spherical grid.\label{fig:checkerboard}}
\end{figure}
The timings are performed using an AMD dual EPYC 7742 with 128 cores and with 1024GB memory.

\subsection{Application of $(\M-\K)^{-1}$ }\label{sec:apply_MK_inv}
We show that $(\bMo-\bKo)^{-1}$ can be applied efficiently and robustly in $g$. To that end, we implemented a preconditioned conjugate gradient method with preconditioner $\bMo-\bKo_N$, see \cref{sec:pre_M-K}. 
\cref{tab:apply_MK_inv} shows the required iteration counts to achieve a relative error below $10^{-13}$. For all $g$, the iteration counts decrease with $N$ as predicted by the considerations in \cref{sec:complexity}.
In particular, since $\bKo=\bKo_N=0$, only one iteration is needed for convergence for $g=0$.
Moreover, we see that, although increasing the value of $N$ increases the workload per iteration, the overall solution time can decrease, which is due to the fact that the scattering operator dominates the computational cost for moderate $d_N$, see \cref{sec:complexity}.
In the remainder of the paper, we employ $N=5$, which yields fast convergence for the considered values of $g$.
\begin{table}[ht]
\centering \footnotesize
\setlength{\tabcolsep}{8pt} 
\renewcommand{\arraystretch}{1.2} 
\caption{Iteration counts (timings in sec.) for the application of $(\bMo-\bKo)^{-1}$ using a preconditioned CG method with preconditioner $\bMo-\bKo_N$ and tolerance $10^{-13}$ for $n_S^+=256$ and $n_R^+=12\,769$.\label{tab:apply_MK_inv}}
\begin{tabular}{ r r c c c c c c c c c}
    \toprule
     && \multicolumn{6}{c}{$g$}\\
    	\cmidrule{3-8}	
    $N$ & $d_N$   & 0 & 0.1     & 0.3     & 0.5      & 0.7       & 0.9\\
	\midrule
    -1 &  0 & 1 (1.6) & 4 (4.2) & 6 (6.1) & 8 (8.0)  & 11 (10.7) & 21 (19.9)\\
	1  &  3 & 1 (1.6) & 3 (3.3) & 5 (4.9) & 7 (6.7)  & 10 (9.5)  & 19 (17.4)\\
	3  & 10 & 1 (1.7) & 2 (2.6) & 5 (5.0) & 6 (6.1)  &  8 (7.8)  & 19 (17.6)\\
	5  & 21 & 1 (1.8) & 2 (2.7) & 3 (3.6) & 4 (4.5)  &  7 (7.1)  & 15 (14.2)\\
	7  & 36 & 1 (1.8) & 2 (2.8) & 3 (3.6) & 4 (4.6)  &  6 (6.4)  & 14 (13.4)\\
	9  & 55 & 1 (1.9) & 2 (2.8) & 2 (2.8) & 4 (4.7)  &  6 (6.4)  & 12 (12.1)\\
	\bottomrule
\end{tabular}
\end{table}

\subsection{Convergence rates}
We study the norm $\eta$ of the iteration matrix $(\bIde-\bP_G)(\bIde-\bP_1^l(\bE-\bKe))$ defined in \cref{eq:def_eta} and its spectral radius 
\begin{align*}
    \rho = \max\{|\lambda|:\,\,  \lambda \text{ is an eigenvalue of } (\bIde-\bP_G)(\bIde-\bP_1^l(\bE-\bKe))\}
\end{align*}
for different choices of preconditioners $\bP_1=\bP_1^l$, anisotropy factors $g$ and dimensions $d_N$ chosen for the subspace correction. 
Since $\bP_G$ is a projection, we have that
\begin{align*}
    (\bIde-\bP_G)^\intercal (\bE-\bKe) (\bIde-\bP_G) = (\bE-\bKe)(\bIde-\bP_G).
\end{align*}
Therefore, $\eta^2$ is the largest eigenvalue of the eigenvalue problem
\begin{align*}
    (\bIde-\bP_1^l(\bE-\bKe))(\bIde-\bP_G)(\bIde-\bP_1^l(\bE-\bKe)) \bw=\lambda \bw.
\end{align*}
We use \textsc{Matlab}'s eigs function to compute $\rho$ and $\eta$ with tolerance $10^{-7}$ and maximum iterations set to $300$.

For the isotropic case $g=0$, $\bP_1^l=\bE_0^{-1}=\bE^{-1}$, i.e., $\rho$ and $\eta$ do not depend on $l$.
For $N=0$, \cref{tab:isotropic} shows that the values of $\eta$ and $\rho$ are essentially independent of the discretization parameters, see also \cite{palii2020convergent}. We observed numerically that choosing $N\in\{2,4\}$ improves the values of $\rho$ and $\eta$ only slightly.

\begin{table}[ht]
\centering \footnotesize
\setlength{\tabcolsep}{8pt} 
\renewcommand{\arraystretch}{1.2} 
\caption{Values of $\rho$ and $\eta$ of the iteration matrix for $g=0$ and different angular grids.\label{tab:isotropic}}
\begin{tabular}{c c c c c c}
	\toprule
    $n_S^+$ &  16    & 64    & 256   & 1024  & 4096\\
    \midrule
    $\eta$  & 0.385  & 0.429 & 0.445 & 0.450 & 0.451\\
    $\rho$  & 0.212  & 0.261 & 0.280 & 0.286 & 0.288\\
	\bottomrule
\end{tabular}
\end{table}

In the next experiments, we vary $g$ from $0.1$ to $0.9$ in steps of $0.2$. \cref{tab:g01}--\cref{tab:g09} display the corresponding values of $\rho$ and $\eta$. For these anisotropic cases, the iteration count $l$ for the preconditioner $\bP_1^l$ as well as the number $d_N$ defined in \cref{sec:subspace_correction} play an important role.
For all combinations of $d_N$ and $l$, we observe a convergent behavior with $\eta\leq c<1$, which is in line with \cref{lem:preconditioner_pl}. 
The values of $\rho$ and $\eta$ decrease substantially with increasing $d_N$ which is inline with the motivation of \cref{sec:motivation}, while, for fixed $d_N$ a saturation in $l$ can be observed.
For $d_N$ sufficiently large, it seems that $\rho=\eta=g^{l}$, see, e.g. \cref{tab:g07} for $d_4$ and $1\leq l\leq 4$.
We may conclude that we can achieve very good convergence rates for moderate values of $d_N$ and $l$ if combined appropriately.

\begin{table}[ht]
\centering \footnotesize
\setlength{\tabcolsep}{8pt} 
\renewcommand{\arraystretch}{1.2} 
\caption{Values of $\rho$ and $\eta$ for $g=0.1$ and different values of $d_N$ and $l$ to realize $\bP_1^l$.\label{tab:g01}}
\begin{tabular}{c c c c c c c c c} 
	\toprule
    & \multicolumn{2}{c}{$d_0=1$} &&\multicolumn{2}{c}{$d_2=6$} &&\multicolumn{2}{c}{$d_4=15$}\\
    \cmidrule{2-3}	\cmidrule{5-6}	\cmidrule{8-9}
    $l$ &  $\rho$ &$\eta$ &&$\rho$ &$\eta$ &&$\rho$ &$\eta$\\
	\midrule
	1 &0.298 & 0.432 && 0.156 & 0.247 && 0.117 & 0.161\\
	2 &0.264 & 0.429 && 0.101 & 0.237 && 0.048 & 0.141\\
	3 &0.261 & 0.429 && 0.097 & 0.237 && 0.043 & 0.141\\
	4 &0.261 & 0.429 && 0.097 & 0.237 && 0.042 & 0.141\\
	5 &0.261 & 0.429 && 0.097 & 0.237 && 0.042 & 0.141\\
    6 &0.261 & 0.429 && 0.097 & 0.237 && 0.042 & 0.141\\
	\bottomrule
\end{tabular}
\end{table}

\begin{table}[ht]
\centering \footnotesize
\setlength{\tabcolsep}{8pt} 
\renewcommand{\arraystretch}{1.2} 
\caption{Values of $\rho$ and $\eta$ for $g=0.3$ and different values of $d_N$ and $l$ to realize $\bP_1^l$.\label{tab:g03}}
\begin{tabular}{c c c c c c c c c} 
	\toprule
  & \multicolumn{2}{c}{$d_0=1$} &&\multicolumn{2}{c}{$d_2=6$} &&\multicolumn{2}{c}{$d_4=15$}\\
    \cmidrule{2-3}	\cmidrule{5-6}	\cmidrule{8-9}
    $l$ &  $\rho$ &$\eta$ &&$\rho$ &$\eta$ &&$\rho$ &$\eta$\\
	\midrule
	1 &0.392 & 0.473 && 0.311 & 0.332 && 0.300 & 0.302  \\
	2 &0.299 & 0.448 && 0.146 & 0.246 && 0.106 & 0.157  \\
	3 &0.284 & 0.447 && 0.111 & 0.242 && 0.060 & 0.146  \\
	4 &0.281 & 0.447 && 0.103 & 0.242 && 0.050 & 0.146  \\
	5 &0.280 & 0.447 && 0.101 & 0.242 && 0.047 & 0.146  \\
    6 &0.280 & 0.447 && 0.101 & 0.242 && 0.046 & 0.146  \\
	\bottomrule
\end{tabular}
\end{table}
\begin{table}[ht]
\centering \footnotesize
\setlength{\tabcolsep}{8pt} 
\renewcommand{\arraystretch}{1.2} 
\caption{Values of $\rho$ and $\eta$ for $g=0.5$ and different values of $d_N$ and $l$ to realize $\bP_1^l$.\label{tab:g05}}
\begin{tabular}{c c c c c c c c c} 
	\toprule
  & \multicolumn{2}{c}{$d_0=1$} &&\multicolumn{2}{c}{$d_2=6$} &&\multicolumn{2}{c}{$d_4=15$}\\
    \cmidrule{2-3}	\cmidrule{5-6}	\cmidrule{8-9}
    $l$ &  $\rho$ &$\eta$ &&$\rho$ &$\eta$ &&$\rho$ &$\eta$\\
	\midrule
	1 & 0.522 & 0.553 && 0.499 & 0.499 && 0.499 & 0.499 \\
	2 & 0.386 & 0.489 && 0.265 & 0.301 && 0.250 & 0.255 \\
	3 & 0.361 & 0.482 && 0.174 & 0.260 && 0.136 & 0.175 \\
	4 & 0.358 & 0.480 && 0.147 & 0.254 && 0.089 & 0.159 \\
	5 & 0.357 & 0.480 && 0.140 & 0.253 && 0.070 & 0.156 \\
    6 & 0.357 & 0.480 && 0.137 & 0.253 && 0.062 & 0.156 \\
	\bottomrule
\end{tabular}
\end{table}
\begin{table}[ht]
\centering \footnotesize
\setlength{\tabcolsep}{8pt} 
\renewcommand{\arraystretch}{1.2} 
\caption{Values of $\rho$ and $\eta$ for $g=0.7$ and different values of $d_N$ and $l$ to realize $\bP_1^l$. The symbol $-$ indicates that \textsc{Matlab}'s eigs function has not converged to the desired tolerance.\label{tab:g07}}
\begin{tabular}{c c c c c c c c c} 
	\toprule
  & \multicolumn{2}{c}{$d_0=1$} &&\multicolumn{2}{c}{$d_2=6$} &&\multicolumn{2}{c}{$d_4=15$}\\
    \cmidrule{2-3}	\cmidrule{5-6}	\cmidrule{8-9}
    $l$ &  $\rho$ &$\eta$ &&$\rho$ &$\eta$ &&$\rho$ &$\eta$\\
	\midrule
	1 &  ---  & 0.699 && 0.699 & 0.699 && 0.699 & 0.699\\
	2 & 0.537 & 0.582 && 0.489 & 0.489 && 0.489 & 0.489\\
	3 & 0.515 & 0.567 && 0.349 & 0.366 && 0.342 & 0.342\\
	4 & 0.512 & 0.565 && 0.270 & 0.319 && 0.241 & 0.253\\
	5 & 0.511 & 0.564 && 0.248 & 0.309 && 0.178 & 0.212\\
    6 & 0.511 & 0.564 && 0.239 & 0.306 && 0.142 & 0.195\\
	\bottomrule
\end{tabular}
\end{table}
\begin{table}[ht]
\centering \footnotesize
\setlength{\tabcolsep}{8pt} 
\renewcommand{\arraystretch}{1.2} 
\caption{Values of $\rho$ and $\eta$ for $g=0.9$ and different values of $d_N$ and $l$ to realize $\bP_1^l$. The symbol $-$ indicates that \textsc{Matlab}'s eigs function has not converged to the desired tolerance.\label{tab:g09}}
\begin{tabular}{c c c c c c c c c}
	\toprule
   & \multicolumn{2}{c}{$d_0=1$} &&\multicolumn{2}{c}{$d_2=6$} &&\multicolumn{2}{c}{$d_4=15$}\\
    \cmidrule{2-3}	\cmidrule{5-6}	\cmidrule{8-9}
    $l$ &  $\rho$ &$\eta$ &&$\rho$ &$\eta$ &&$\rho$ &$\eta$\\
	\midrule
	1 & ---  & ---  && ---  &---   && ---&0.899\\
	2 & 0.808&0.808 &&0.808 &0.808 && 0.808&0.808 \\
	3 & 0.764&0.775 &&0.758 &0.758 && 0.727&0.727 \\
	4 & 0.763&0.773 &&0.757 &0.757 && 0.653&0.653 \\
	5 & 0.763&0.772 &&0.757 &0.757 && 0.587&0.587 \\
	6 & 0.763&0.772 &&0.757 &0.757 && 0.528&0.528 \\
	\bottomrule
\end{tabular}
\end{table}

\subsection{$\mathcal{H}^2$-matrix approximation of $\S$}
We demonstrate the $\mathcal{H}^2$-com\-press\-i\-bi\-li\-ty of the scattering operator $\S$. Since every $\mathcal{H}^2$-matrix can be represented as an $\mathcal{H}$-matrix, this also demonstrates the compressibility of $\S$ by means of $\mathcal{H}$-matrices. For the implementation we use a \textsc{Mex} interface to include the library \textsc{H2Lib} \cite{Boe} into our \textsc{Matlab}-implementation.

For the numerical experiments themselves, we choose $g=0.5$ and the same quadrature formula in our \textsc{Matlab} implementation and in our implementation within the \textsc{H2Lib}. The compression algorithm of $\textsc{H2Lib}$ uses multivariate polynomial interpolation, requiring the extension of the Henyey-Greenstein kernel as in \cref{eq:HGextension}. The compression parameters are set to an admissibility parameter
$\eta_{\H}=1.4$,
$p=4$ interpolation points on a single interval and a minimal block size parameter $n_{\min}=64$, see \cite{Bor2010,Hac2015}. We also tested an implementation without the need for an extension within the \textsc{Bembel} library \cite{DHK+2020} which yields similar results, but requires a finite element discretization on quadrilaterals, rather than triangles.
In both cases, the differences between dense and compressed scattering matrix are below the discretization error.

\cref{tab:comparison_compression} lists the memory requirements, setup time, and time for a single matrix-vector multiplication of $\bSe$ in dense and $\mathcal{H}^2$-compressed form. We can clearly observe the quadratic complexity for storage and matrix-vector multiplication of the dense matrices and the asymptotically linear complexity of the $\mathcal{H}^2$-matrices. The scaling of the assembly times for dense and $\mathcal{H}^2$-matrices seems to be worse than predicted by theory, which is possibly caused by memory issues. Nevertheless, the scaling of the $\mathcal{H}^2$-matrices for the assembly times is much better than the one for dense matrices.

\begin{table}[ht]
\centering \footnotesize
\setlength{\tabcolsep}{8pt} 
\renewcommand{\arraystretch}{1.2} 
\caption{Memory consumption in MB, timings in sec.~for assembly and matrix-vector multiplication of $\bSe$ and corresponding $\H^2$-matrix approximation $\overline{\bSe}$ for $g=0.5$. Numbers in brackets indicate the ratio to the previous refinement level.\label{tab:comparison_compression}}
\begin{tabular}{r c c c c c c}
	\toprule
	$n_S^+$ & mem $\bSe$ & setup $\bSe$ & apply $\bSe$\\
	\midrule
	$     64$ & $0.0312$          & $0.171$          & $6.9\cdot 10^{-5}$         \\
	$    256$ & $0.5$ ($16.0$) & $0.203$ ($1.2$) & $6.5\cdot 10^{-5}$ ($0.9$)\\
	$   1\,024$ & $8$ ($16.0$) & $0.438$ ($2.2$) & $0.000313$ ($4.8$)\\
	$   4\,096$ & $128$ ($16.0$) & $4.2$ ($9.6$) & $0.00517$ ($16.5$)\\
	$  16\,384$ & $2.05\cdot 10^3$ ($16.0$) & $189$ ($45.0$) & $0.0805$ ($15.6$)\\
	$  65\,536$ & $3.28\cdot 10^4$ ($16.0$) & $1.09\cdot 10^4$ ($57.5$) & $2.67$ ($33.1$)\\
	$ 262\,144$ &                 --- &                 --- &                 ---\\
	$1\,048\,576$ &                 --- &                 --- &                 ---\\
	\midrule
	$n_S^+$ & mem $\overline{\bSe}$ & setup $\overline{\bSe}$ & apply $\overline {\bSe}$\\
	\midrule
	$     64$ & $0.0313$          & $0.00109$          & $0.00025$         \\
	$    256$ & $0.502$ ($16.0$) & $0.0116$ ($10.7$) & $0.000547$ ($2.2$)\\
	$   1\,024$ & $11.3$ ($22.5$) & $0.139$ ($11.9$) & $0.0086$ ($15.7$)\\
	$   4\,096$ & $89.2$ ($7.9$) & $0.902$ ($6.5$) & $0.0841$ ($9.8$)\\
	$  16\,384$ & $484$ ($5.4$) & $4.75$ ($5.3$) & $0.328$ ($3.9$)\\
	$  65\,536$ & $2.27\cdot 10^3$ ($4.7$) & $24.6$ ($5.2$) & $1.46$ ($4.4$)\\
	$ 262\,144$ & $9.53\cdot 10^3$ ($4.2$) & $182$ ($7.4$) & $6.92$ ($4.7$)\\
	$1\,048\,576$ & $3.82\cdot 10^4$ ($4.0$) & $1.46\cdot 10^3$ ($8.0$) & $28.5$ ($4.1$)\\
	\bottomrule
\end{tabular}
\end{table}

\subsection{Benchmark example}
The viability of the preconditioned Richardson iteration \cref{eq:Richardson_discrete} is shown for some larger computations.
We fix $g=0.5$ and solve the even-parity equations \cref{eq:ep_discrete} for the lattice problem.
We fix $l=4$ steps to realize the preconditioner $\bP_1^l$ and $N=4$, i.e., we use $d_4=15$ eigenfunctions of $\bSe$ for the subspace correction, cf. \cref{sec:subspace_correction}. 
In view of \cref{tab:g05}, we expect a contraction rate $\eta\approx 0.16$. Therefore, in order to achieve an error bound $\|\bue-\bue_{n}\|_{\bE-\bKe}<10^{-8}$, we expect to require $n\approx 10$ iterations.
In our implementation, we choose $\bue_0=0$, and we stop the iteration at index $n$ for which
\begin{align}\label{eq:stop}
 \| \bue_{n} - \bue_{n-1}\|_{\bE-\bKe}<10^{-8} \| \bue_{1}\|_{\bE-\bKe}.
\end{align}
Note that, assuming a contraction rate $\eta=0.16$, Banach's fixed point theorem asserts that the error satisfies $\|\bue-\bue_{n}\|_{\bE-\bKe} \leq 0.2 \| \bue_{n} - \bue_{n-1}\|_{\bE-\bKe}$.
The dimension of the problem on the finest grid is $n_R^+n_S^+=207\,360\,000$, i.e., storing the solution vector requires $1.5$GB of memory. Note that the corresponding dimension of the solution vector to the mixed system is about $1.5\times 10^9$.
Motivated by \cref{tab:comparison_compression} we implement the scattering operators $\bSe$ and $\bSo$ using dense matrices in this example.
The application of $\bE_0^{-1}$ is implemented with \textsc{Matlab}'s sparse LU factorization, i.e., here, $\gamma\leq 1.5$ in the complexity estimates of \cref{sec:complexity}.

\cref{tab:benchmark} displays the iteration counts and timings for different grid refinements.
We observe mesh-independent convergence behavior of the iteration which matches well the theoretical bound $n\approx 10$. 
Furthermore, the computation time scales like $(n_R^+)^{1.3}$ for fixed $n_S^+$.
If $n_S^+$ increases from $1024$ to $4096$, the superlinear growth in computation time can be explained by using dense matrices for $\bSe$ and $\bSo$, which, as shown in \cref{tab:comparison_compression}, can be remedied by using the compressed scattering operators. 

\begin{table}[ht]
\centering \footnotesize
\setlength{\tabcolsep}{8pt} 
\renewcommand{\arraystretch}{1.2} 
\caption{Iteration index $n$ (timings in sec.) such that \cref{eq:stop} holds for the benchmark example.\label{tab:benchmark}}
\begin{tabular}{r c c c } 
	\toprule
   & \multicolumn{3}{c}{$n_R^+$}\\
    \cmidrule{2-4}
    $n_S^+$ &  $3\,249$ & $12\,769$ &$50\,625$\\
	\midrule
	  64 & 8  (50)  & 9 (236)\, & 9 (1\,470)\, \\
	 256 & 9  (114) & 9 (499)\, & 9 (2\,476)\,\\
	1\,024 & 9  (300) & 9 (1\,107) & 10 (6\,580)\,\\
	4\,096 & 9 (1\,017) & 9 (4\,983) & 10 (34\,029)\\
	\bottomrule
\end{tabular}
\end{table}

\section{Conclusions}\label{sec:discussion}
We have presented efficient preconditioned Richardson iterations for anisotropic radiative transfer that are provably convergent and show robust convergence in the optical parameters, which comprises forwarded peaked scattering and heterogeneous absorption and scattering coefficients. 
This has been achieved by employing black-box matrix compression techniques to handle the scattering operator efficiently, and by construction of appropriate preconditioners. In particular, we have shown that, for anisotropic scattering, subspace corrections constructed from low-order spherical harmonics expansions considerably improve the convergence of our iteration.

On the discrete level, our preconditioners can be obtained algebraically from the matrices of any FEM code providing the matrices from the mixed system \cref{eq:mixed_discrete}. We discussed further implementational details and their computational complexity, which, supported by several numerical tests, showed the efficiency of our method. If a solver with linear computational complexity for anisotropic elliptic problems is employed to realize $\bE_0^{-1}$, each single iteration of our scheme has linear computational complexity in the discretization parameters.
Our numerical examples employed low-order polynomials for discretization, but the presented methodology directly applies to high-order polynomial approximations as well.

Finally, let us mention that the saddle-point problem \cref{eq:mixed_discrete_intro} may also be solved using the MINRES algorithm after appropriate multiplication of the second equation by $-1$. In view of the inf-sup theory for \cref{eq:op1}--\cref{eq:op2} given in \cite{egger2012mixed}, block-diagonal preconditioners with blocks $\bE-\bKe$ and $\bMo-\bKo$ lead to robust convergence behavior \cite[Section 5.2]{Wathen_2015}, but the efficient inversion of $\bE-\bKe$ is as difficult as solving the even-parity equations, which has been considered in this paper.

\bibliographystyle{siamplain} 
\bibliography{dsa_anisotropic}
\end{document}